\documentclass{article}

\usepackage[margin=1in]{geometry}
\usepackage[parfill]{parskip}
\usepackage[utf8]{inputenc}
\usepackage{amsmath,amssymb,amsfonts,amsthm,mathrsfs,color}
\usepackage{graphicx}
\usepackage[skins]{tcolorbox}
\usepackage{bm}
\usepackage{multirow}
\usepackage{subcaption}
\usepackage[colorlinks=true,linkcolor=blue,citecolor=blue]{hyperref} 

\newtheorem{Theorem}{Theorem}[section]

\newtheorem{Remark}[Theorem]{Remark}  
\newtheorem{Assumption}[Theorem]{Assumption}

\makeatletter
\renewcommand{\fnum@figure}{Fig. \thefigure}
\makeatother

\newcommand\hl{\color{orange}\bf}

\let\OLDthebibliography\thebibliography
\renewcommand\thebibliography[1]{
  \OLDthebibliography{#1}
  \setlength{\parskip}{1pt}
  \setlength{\itemsep}{0pt plus 0.0ex}}

\title{Stein's Method of Moments on the Sphere}
\author{Adrian Fischer\footnote{Adrian Fischer, University of Oxford, UK. E-mail: adrian.fischer@stats.ox.ac.uk}, Robert E. Gaunt\footnote{Robert E. Gaunt, The University of Manchester, UK. E-mail: robert.gaunt@manchester.ac.uk}\: and Yvik Swan\footnote{
Yvik Swan, Université libre de Bruxelles, Belgium. E-mail: yvik.swan@ulb.be}}

\begin{document}
\maketitle

\begin{abstract}
    We use Stein characterizations  to obtain new moment-type  estimators for the parameters of three classical spherical distributions (namely the Fisher-Bingham, the von Mises-Fisher, and the Watson distributions)  in the i.i.d.\ case. This leads to explicit estimators which  have good asymptotic properties (close to efficiency) and therefore lead to interesting alternatives to classical   maximum likelihood methods or more recent score matching estimators. We perform competitive simulation studies to assess the quality of the new estimators.  Finally, the practical relevance of our estimators  is illustrated on a real data application  in spherical latent representations of handwritten numbers.
\end{abstract}


\section{Introduction}

Directional statistics is the branch of  statistics that deals with   data that comes in the form of  angles, directions, coordinates, etc.  Areas of application which have  driven the research in this field include astronomy \cite{garcia2023projection,garcia2020optimal,ghosh2024directional,jupp1995some}, bioinformatics \cite{ameijeiras2022sine,dortet2008model,kent2005using,mardia2018directional}, ecology \cite{ameijeiras2019fire,landler2018circular, merrifield2006statistical}, and machine learning \cite{chen2021maximum,gopal2014mises,sra2018directional,wang2023deep}. This list is by no means exhaustive and we refer to the monographs \cite{fisher1993statistical,ley2017modern, mardia2000directional, pewsey2013circular} and surveys \cite{garcia2018overview,mardia2024fisher, pewsey2021recent} for more applications and references. \par
 
At the heart of directional statistics are  probability  distributions on the hyperspheres  $\mathcal{S}^{d-1}  := \{ x \in \mathbb {R}^d \mbox{ such that } x^\top x = 1\}$ of $\mathbb{R}^d$ for $d \ge 2$. One of the most important families of such  distributions    is   the so-called \emph{Fisher-Bingham} family which is obtained by conditioning  a multivariate normal distribution to lie on a $d$-dimensional hypersphere. This family is characterized by the  probability density function (pdf) 
\begin{align*}
    p_\theta(x) = \frac{1}{C(\mu,A)} \exp\big( \mu^{\top}x+ x^{\top}Ax \big), \quad x \in \mathcal{S}^{d-1}, 
\end{align*} 
where  $\theta = (\mu,A) $ with $\mu \in \mathbb{R}^d$ and $A \in \mathbb{R}^{d \times d}$ symmetric,  $C(\mu,A) = \int_{\mathcal{S}^{d-1}}  \exp( \mu^{\top}x+ x^{\top}Ax) \sigma_{d-1}(dx)$ is the normalising constant and, here and throughout, densities on $\mathcal{S}^{d-1}$ are taken with respect to the surface area measure $\sigma_{d-1}$ on $\mathcal{S}^{d-1}$.    The Fisher-Bingham family includes the \emph{Watson} distribution with density $    p_{\theta}(x) \propto \exp\big(\kappa (\mu^{\top} x)^2\big)$ and the \emph{von Mises-Fisher} distribution (vMF)  with density  $p_{\theta}(x) \propto  \exp\big(\kappa \mu^{\top} x\big)$ (dubbed ``the most important distribution in directional data analysis''  by \cite{mabdia1975distribution}). \par 

The normalising constant for the vMF and Watson distributions can be written in terms of   special functions (gamma, modified Bessel,  confluent hypergeometric); for general parameters the normalising constant  $C(\mu, A)$ of Fisher-Bingham distributions is not known in closed form. Numerical computation of this normalising constant is numerically unstable and computationally heavy, especially in a high-dimensional setting.  Consequently, maximum likelihood estimation of the parameters of these distributions requires developing fast and reliable numerical procedures for evaluating the normalizing constants. See e.g.\ \cite{christie2015efficient,hornik2014movmf,hornik2014maximum,sra2012short,tanabe2007parameter} for estimation of the parameters of vMF distribution,  \cite{dey2022inference, sra2013multivariate} for those of the Watson distribution, and   \cite{chen2021maximum,koyama2014holonomic,kume2018exact,kume2005saddlepoint} for those of the Fisher-Bingham family of distribution. A natural alternative to maximum likelihood methods is the score matching approach \cite{hyvarinen2007some,hyvarinen2005estimation}, which is famously applicable to unnormalized models. This approach was extended to manifolds (and thus also  hyperspheres) in \cite{mardia2016score, williams2022score}. However, in concrete examples the resulting estimators still have to be computed via numerical optimization which can be a problem for distributions on the hypersphere such as the Fisher-Bingham family where the dimension of the parameter space grows quadratically with the dimension of the data. Hence, for more complicated models, the score matching approach is often applied to settings where some of the parameters are assumed to be known or can be estimated through other techniques. \par 

In this paper, we pursue a different approach based on Stein identities for spherical distribution. Our work is  an extension of \cite{ebner2023point} in which the authors used the characterising property of \textit{Stein operators} to perform parameter estimation in univariate parametric models. We quickly sketch the approach in a multivariate (Euclidean) setting. Let $X$ be a random vector that lives in $\mathbb{R}^d$  and admits a differentiable probability density function (pdf) $p_{\theta}$ which depends on a parameter $\theta \in \Theta \subset \mathbb{R}^p$. Then define the \textit{density approach} \cite{ley2017stein,ley2013stein,mijoule2023stein} Stein operator as 
\begin{align} \label{def_mult_stein operator}
    \mathcal{A}_{\theta}f(x)= \frac{\nabla \big( f(x) p_{\theta}(x) \big)}{p_{\theta}(x)},
\end{align}
where $f:\mathbb{R}^d \rightarrow \mathbb{R}$ is a differentiable test function. Operator \eqref{def_mult_stein operator} does not depend on the normalising constant of $p_\theta$ and, crucially, satisfies 
\begin{align} \label{Stein_equation_mult}
    \mathbb{E} [ \mathcal{A}_{\theta}f(X)] =0
\end{align}
for all functions $f: \mathbb R^d   \to \mathbb R$ belonging to a wide (and explicit)   class $\mathscr{F}_{\theta}$. We call \eqref{Stein_equation_mult} a  \textit{Stein identity} for $p_\theta$. Under general assumptions on $p_\theta$, by some appropriate law of large numbers it should then be the case that, with $X_1, \ldots, X_n$ an i.i.d.\ sample from $p_\theta$, 
$  \overline{\mathcal{A}_{\theta}f(X)} := n^{-1} \sum_{i=1}^n \mathcal{A}_{\theta}f(X_i) \approx 0$ in probability for $n$ sufficiently large and  any $f \in \mathscr{F}_\theta$. 
Hence for a $p$-dimensional parameter $\theta \in \Theta \subset \mathbb{R}^p$, fixing  $p$  test functions $f_1, \ldots, f_p$ in $\mathscr{F}= \cap_{\theta \in \Theta} \mathscr{F}_\theta$ and solving for $\theta$ the system of $p$ equations    in $p$ unknowns 
\begin{equation}\label{Stein_equation_multemp}
    \overline{\mathcal{A}_{\theta}f_j(X)} = 0 \mbox{ for all  } j = 1, \ldots, p
\end{equation}
  leads to a moment type estimator for $\theta$, which we call a  \textit{Stein estimator}. We call \eqref{Stein_equation_multemp} the empirical Stein equation, and the whole approach \emph{Stein's method of moments}. \par

The purpose of this paper is to develop Stein's method of moments on hyperspheres.  To this end, we need a suitable counterpart to \eqref{def_mult_stein operator} and \eqref{Stein_equation_mult}. Before specialising to spheres, let $\mathcal{M}$ be a smooth compact $(d-1)$-dimensional sub-manifold of $\mathbb{R}^d$ (with or without boundary) equipped with the Riemannian metric $\langle \cdot, \cdot \rangle_{\mathcal{M}}$ and let $p_{\theta}:\mathcal{M} \rightarrow \mathbb{R}$ be the smooth density of a probability distribution on $\mathcal{M}$ (with respect to the surface area measure on $\mathcal{M}$) where $\theta \in \Theta \subset \mathbb{R}^p$ is the parameter we wish to estimate. Moreover, let
\begin{align*}
    \mathscr{F}:=\big\{ f: \mathcal{M} \rightarrow \mathbb{R} \, \vert \, f \text{ is smooth} \big\}.
\end{align*}
Green's first identity then ensures that we have for any $f \in \mathscr{F}$,
\begin{align} \label{green_id}
    \int_{\mathcal{M}} p_{\theta} \Delta_{\mathcal{M}} f + \int_{\mathcal{M}} \langle \nabla_{\mathcal{M}} p_{\theta}, \nabla_{\mathcal{M}} f \rangle_{\mathcal{M}} = \int_{\partial \mathcal{M}} p_{\theta} f \vec{n},
\end{align}
where $\vec{n}$ is the outward pointing normal, $\partial \mathcal{M}$ is the manifold boundary and $\Delta_{\mathcal{M}}$, $\nabla_{\mathcal{M}}$ are the Laplace-Beltrami and gradient operator on $\mathcal{M}$. For a manifold without boundary, the right-hand side of \eqref{green_id} is equal to zero which gives rise to the Stein operator
\begin{align} \label{stein_operator_manifold}
    \mathcal{A}_{\theta}f(x)= \Delta_{\mathcal{M}} f(x) + \Big\langle \frac{\nabla_{\mathcal{M}} p_{\theta}(x)}{p_{\theta}(x)}, \nabla_{\mathcal{M}} f(x) \Big\rangle_{\mathcal{M}}
\end{align}
for any distribution on a manifold without boundary $\mathcal{M}$ with smooth density $p_{\theta}$ and any smooth function $f$. This operator is independent of a parametrization of $\mathcal{M}$ and was recently used in \cite{le2024diffusion}, in which the authors developed bounds on the Wasserstein distance between distributions on Riemannian manifolds. The operator \eqref{stein_operator_manifold} is the manifold equivalent to \eqref{def_mult_stein operator} applied to the gradient of a test function $f$. 

Let us now consider the unit sphere $\mathcal{S}^{d-1}$ endowed with the canonical metric $\langle \cdot, \cdot \rangle_{\mathcal{S}}$. In order to calculate the Stein operator, we wish to express all differential operators above in terms of their Euclidean counterparts (the same applies to the Riemannian metric). Therefore, let $\Delta$, $\nabla$, $\nabla^2$ be the standard operators on $\mathbb{R}^d$. Note that every smooth function $f:\mathcal{M} \rightarrow \mathbb{R}$ can be extended to a smooth function $\tilde{f}$ defined on an open set $U \subset \mathbb{R}^d$ with $\mathcal{S}^{d-1} \subset U$ (compare \cite[Lemma 5.34]{lee2012introduction}). In the sequel, $\Delta f$, $\nabla f$ and $\nabla^2 f$ then correspond to $\Delta \tilde{f}$, $\nabla \tilde{f}$ and $\nabla^2 \tilde{f}$. For a function $f \in \mathscr{F}$ we can express the spherical Laplace-Beltrami operator $\Delta_{\mathcal{S}}$ in terms of the Euclidean Laplace operator. We have
\begin{align*}
    \Delta_{\mathcal{S}} f(x) = (1-d) \ x^{\top} \nabla f(x) - x^{\top} \nabla^2 f(x) x + \Delta f(x), \quad x \in \mathcal{S}^{d-1}.
\end{align*}
Moreover, one can show that
\begin{align*}
   \Big \langle \frac{\nabla_{\mathcal{S}} p_{\theta}(x)}{p_{\theta}(x)}, \nabla_{\mathcal{S}} f(x) \Big\rangle_{\mathcal{S}} =  \frac{\nabla p_{\theta}(x)}{p_{\theta}(x)}  \big (I_d - xx^{\top} \big) \nabla f(x), \quad x \in \mathcal{S}^{d-1}.
\end{align*}
Note that, in both equations above, the right-hand side is independent of the choice of the smooth extensions of the functions $f$ and $p_{\theta}$. For a distribution on $\mathcal{S}^{d-1}$ with density $p_{\theta}$ we therefore have the Stein operator
\begin{align} \label{stein_operator_manifolds}
    \mathcal{A}_{\theta}f(x)= (1-d) \ x^{\top} \nabla f(x) - x^{\top} \nabla^2 f(x) x + \Delta f(x)+ \frac{\nabla p_{\theta}(x)}{p_{\theta}(x)}  \big (I_d - xx^{\top} \big) \nabla f(x).
\end{align}
It follows immediately from Green's identity that we have indeed
\begin{align} \label{Stein_identity_manifolds}
    \mathbb{E}[\mathcal{A}_{\theta}f(X)]=0
\end{align}
for a random variable $X$ with values in $\mathcal{S}^{d-1}$ which admits the density $p_{\theta}$ and any function $f \in \mathscr{F}$. 
 \par 

The rest of the  paper is organized as follows. In Section \ref{section_fisherbingham}, we propose a new estimator for the Fisher-Bingham distribution and in Sections \ref{section_fisher_von_mises} and \ref{section_watson} we treat the vMF and Watson sub-families, whereby we prove the competitiveness of our approach by comparing to the asymptotically efficient maximum likelihood estimator (MLE). In Section \ref{section_real_data_example}, we apply our new estimator for the Fisher-Bingham distribution to a real data example from machine learning. Finally, Appendix \ref{appendix} collects all proofs.

\section{Fisher-Bingham distribution} \label{section_fisherbingham}
The pdf of the Fisher-Bingham distribution $FB(\mu,A)$ with parameter $\theta = (\mu,A)$ such that $\mu \in \mathbb{R}^{d}$, $A \in \mathbb{R}^{d \times d}$ symmetric, is given by
\begin{align*}
p_{\theta}(x)= \frac{1}{C(\mu,A)} \exp\big( \mu^{\top}x+ x^{\top}Ax \big), \quad x \in \mathcal{S}^{d-1},
\end{align*} 
with unknown normalising constant $C(\mu,A)$. Since $x^{\top}x=1$, we need to impose an additional restriction on $A$ to ensure that the parameters are identifiable. In order to simplify the notation concerning the parameter estimation, we suppose here that $A_{d,d}=0$. If $A$ is positive definite, the Fisher-Bingham distribution can be interpreted as a multivariate Gaussian distribution on $\mathbb{R}^d$ conditioned with respect to the sphere $\mathcal{S}^{d-1}$. \par
The Stein operator \eqref{stein_operator_manifolds} reads
\begin{align*}
    \mathcal{A}f(x)= (1-d) \ x^{\top} \nabla f(x) - x^{\top} \nabla^2 f(x) x + \Delta f(x) + (\mu^{\top} + 2x^{\top} A)  \big (I_d - xx^{\top} \big) \nabla f(x), \quad f \in \mathscr{F}.
\end{align*}
Let us now consider the estimation of the parameters $\mu$ and $A$. For this purpose, suppose we have an i.i.d.\ sample $X_1,\ldots,X_n \sim FB(\mu_0,A_0) $ defined on a common probability space $(\Omega,\mathcal{F},\mathbb{P})$. In view of the restrictions with respect to $A$ we have $d+d(d+1)/2-1$ parameters to estimate. \par

We consider two test functions $f_1:\mathcal{S}^{d-1} \rightarrow \mathbb{R}^d$ and $f_2:\mathcal{S}^{d-1} \rightarrow \mathbb{R}^{d(d+1)/2-1}$ for which each component is in $\mathscr{F}$. We solve the empirical version of the Stein identity \eqref{Stein_identity_manifolds} for $\theta$, i.e.
\begin{align} \label{stein_identity_fisherbingham}
     \frac{1}{n} \sum_{i=1}^n \mathcal{A}_{\theta} f_i(X_i)=0, \quad i=1,2,
\end{align}
where $\mathcal{A}_{\theta}$ is applied to each component of $f_i$, $i=1,2$. For both functions $f_i, \, i=1,2$, we write $\nabla f_i$ for the standard Jacobian matrix and $\nabla^2 f_i$ for the matrix whose rows contain the vectorized Hessians and each row represents a component of the function $f_i$. For example, $\nabla^2 f_2$ is then a $(d(d+1)/2-1) \times d^2$--matrix. \par
We introduce the following quantities:
\begin{alignat*}{2}
   M_n&=2\overline{ \big( \big( \nabla f_2(X) \big (I_d - XX^{\top} \big) \big) \otimes X^{\top} \big) {\bf D } } \quad && \in \mathbb{R}^{(d(d+1)/2-1) \times d(d+1)/2}, \\
   D_n&=\overline{ (d-1) \nabla f_2(X) X +  \nabla^2 f_2(X) (X \otimes X) - \Delta f_2(X) } \quad && \in \mathbb{R}^{d(d+1)/2-1}, \\
   E_n&=\overline{ \big( \nabla f_2(X) \big (I_d - XX^{\top} \big) \big) } \quad && \in \mathbb{R}^{(d(d+1)/2-1) \times d}, \\
   G_n&=2\overline{ \big( \big( \nabla f_1(X) \big (I_d - XX^{\top} \big) \big) \otimes X^{\top} \big) {\bf D} } \quad && \in \mathbb{R}^{d \times d(d+1)/2}, \\
    H_n&=\overline{ (d-1) \nabla f_1(X) X + \nabla^2 f_1(X) (X \otimes X) - \Delta f_1(X) } \quad && \in \mathbb{R}^d, \\
    L_n&=\overline{\nabla f_1(X) (I_d - XX^{\top}) } \quad &&\in \mathbb{R}^{d \times d},
\end{alignat*}
where ${ \bf D} \in \mathbb{R}^{d^2 \times d(d+1)/2}$ is the duplication matrix and $\otimes$ is the standard Kronecker product. In the display above we wrote $\overline{f(X)}= n^{-1}\sum_{i=1}^n f(X_i)$. Since $A_{dd}=0$, we have to adjust the quantities above. Therefore, let $M_n'$ and $G_n'$ be the matrices $M_n$ and $G_n$ without the last column. The solutions $\hat{\mu}_n$ and $\hat{A}_n$ to \eqref{stein_identity_fisherbingham} are then given by
\begin{align*}
    \mathrm{vech}'(\hat{A}_n)&= (M_n')^{-1}(D_n- E_n\hat{\mu}_n), \\
    \hat{\mu}_n&= (L_n - G_n' (M_n')^{-1}E_n)^{-1} (H_n - G_n' (M_n')^{-1} D_n),
\end{align*}
where we wrote $\mathrm{vech}'(\hat{A}_n)$ for the vector $\mathrm{vech}(\hat{A}_n)$ without the last component, which is then set to $0$ given the parametrization of the matrix $A$. \par

In the sequel, we work out the conditions on the test functions under which our new estimators are consistent. We introduce the following assumptions. 
\begin{Assumption} \label{Assumption_consistency_fisherbingham}
    Suppose that the matrices $\mathbb{E}[M_1']$ and $\mathbb{E}[L_1] - \mathbb{E}[G_1'] \mathbb{E}[M_1']^{-1}\mathbb{E}[E_1]$ are invertible.
\end{Assumption}

\begin{Theorem} \label{theoem_consistency_fisherbingham}
Suppose Assumption \ref{Assumption_consistency_fisherbingham} holds. Then $\hat{\theta}_n=(\hat{\mu}_n,\hat{A}_n)$ exists with probability converging to one and is strongly consistent in the following sense: There is a set $A \subset \Omega$ with $\mathbb{P}(A)=1$ such that for each $\omega \in A$ there is a $N \in \mathbb{N}$ such that $\hat{\theta}_n$ exists for each $n \geq N$ and 
\begin{align*}
(\hat{\mu}_n,\hat{A}_n)(\omega) \overset{\mathrm{a.s.}}{\longrightarrow} (\mu_0,A_0)
\end{align*}
as $n \rightarrow \infty$.
\end{Theorem}

We found the following choice of test functions convenient: $f_1:\mathcal{S}^{d-1} \rightarrow \mathbb{R}^d$, $x \mapsto x $ and $f_2:\mathcal{S}^{d-1} \rightarrow \mathbb{R}^{d(d+1)/2}$, $x \mapsto \mathrm{vech}(x x^{\top}) $. Since $x \in \mathcal{S}^{d-1}$, we do not gain any further information with the last component of $f_2$, which we therefore delete. Consequently, we arrive at the correct dimension $d(d+1)/2-1$ regarding $f_2$. We remind of the particular ordering of the elements of $\nabla f_1$, $\nabla f_2$, $\nabla^2 f_1$ and $\nabla^2 f_2$. We calculate $\nabla f_1(x)= I_d$, as well as $\nabla^2 f_1(x) =0$ and $\Delta f_1(x)=0$. For $f_2$, we have that 
\begin{align*}
   \nabla f_2(x) = \begin{pmatrix} 2x_1 & & & \\ x_2 & x_1 & & \\ x_3 & & x_1 & \\ \vdots & & & \ddots \\ 0& 2x_2 & &  \\  & x_3 & x_2 &  \\ & \vdots & & \ddots \end{pmatrix} \in \mathbb{R}^{(d(d+1)/2-1) \times d}.
\end{align*}
Let $E_{ij} \in \mathbb{R}^{d \times d}$ be the matrix which is equal to zero except for the elements $(i,j)$ and $(j,i)$ which are equal to one. Then we obtain
\begin{align*}
    \nabla^2f_2(x) = \begin{pmatrix} 2\mathrm{vec}(E_{11})^{\top} \\  \mathrm{vec}(E_{12})^{\top} \\ \vdots \\  \mathrm{vec}(E_{1d})^{\top} \\ 2\mathrm{vec}(E_{22})^{\top} \\  \mathrm{vec}(E_{23})^{\top} \\ \vdots  \end{pmatrix} \in \mathbb{R}^{(d(d+1)/2-1) \times d^2}.
\end{align*}
Furthermore,
\begin{align*}
    \Delta f_2(x) = 2 \, \mathrm{vech}'(I_d) \in  \mathbb{R}^{d(d+1)/2-1}.
\end{align*}

\begin{figure}
    \vspace{-2cm}
    \centering     
    \begin{minipage}[c]{0.27\textwidth}
        \includegraphics[width=\textwidth]{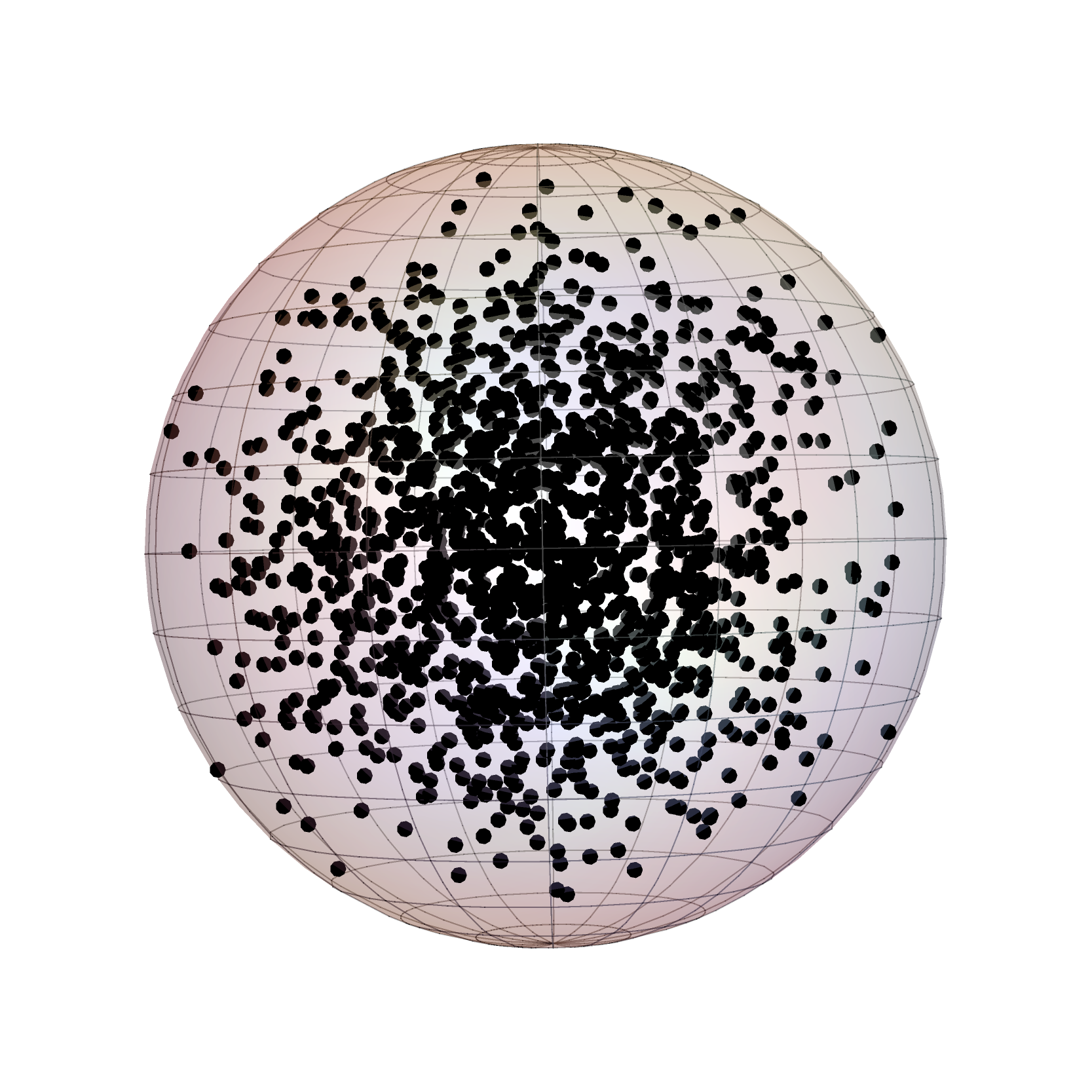}
    \end{minipage}\hfill
    \begin{minipage}[c]{0.22\textwidth}
        \caption{\label{sample1} \\ $\mu=\begin{pmatrix} 10 & 0 & 0 \end{pmatrix}^{\top}$ \\ $A= 0$}
    \end{minipage}
    \begin{minipage}[c]{0.27\textwidth}
        \includegraphics[width=\textwidth]{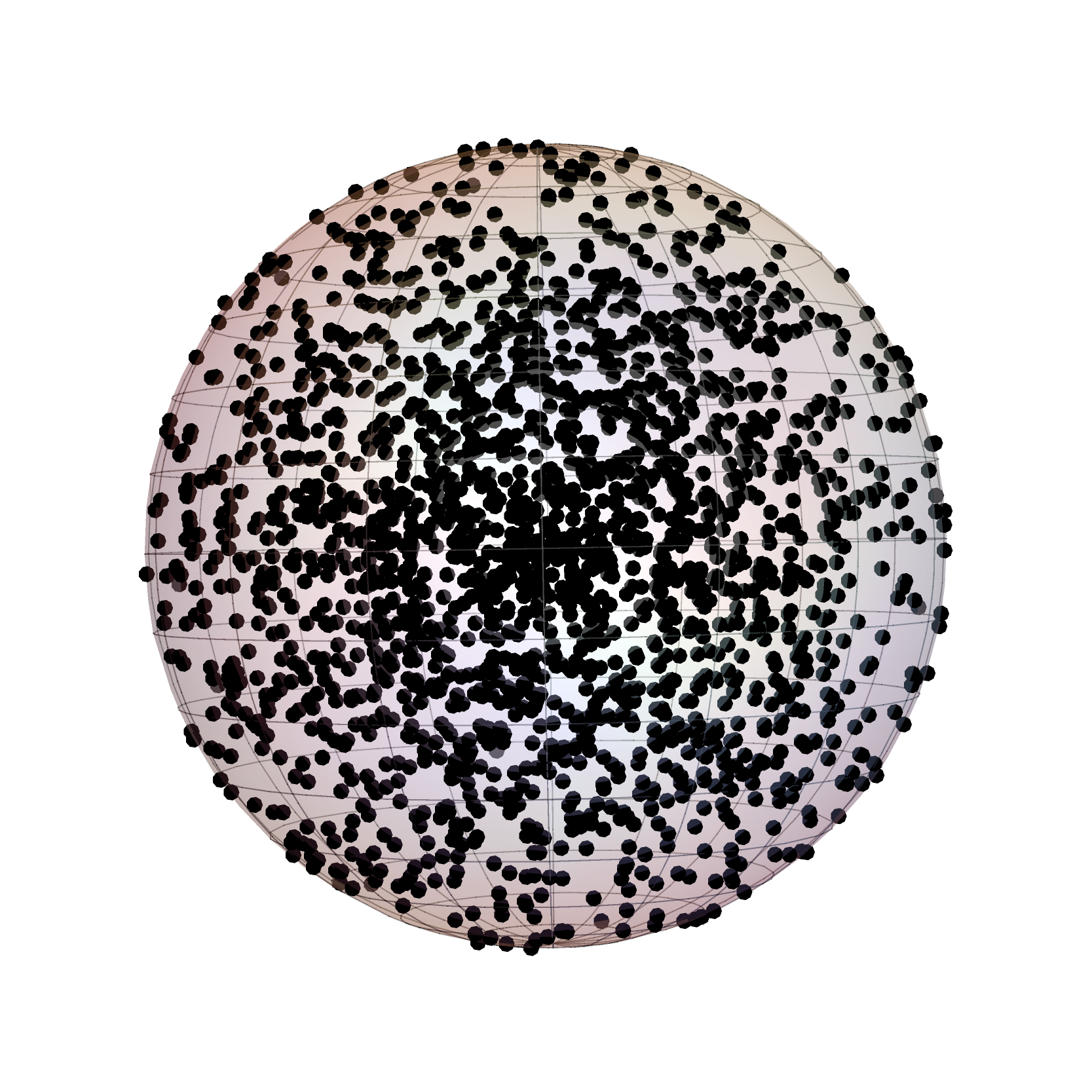}
    \end{minipage} \vspace{-0.5cm} \hfill
    \begin{minipage}[c]{0.22\textwidth}
        \caption{\label{sample2} \\ $\mu=\begin{pmatrix} 5 & 0 & 0 \end{pmatrix}^{\top}$ \\ $A= 0$}
    \end{minipage}
    \vspace{-0.5cm}
    \begin{minipage}[c]{0.27\textwidth}
        \includegraphics[width=\textwidth]{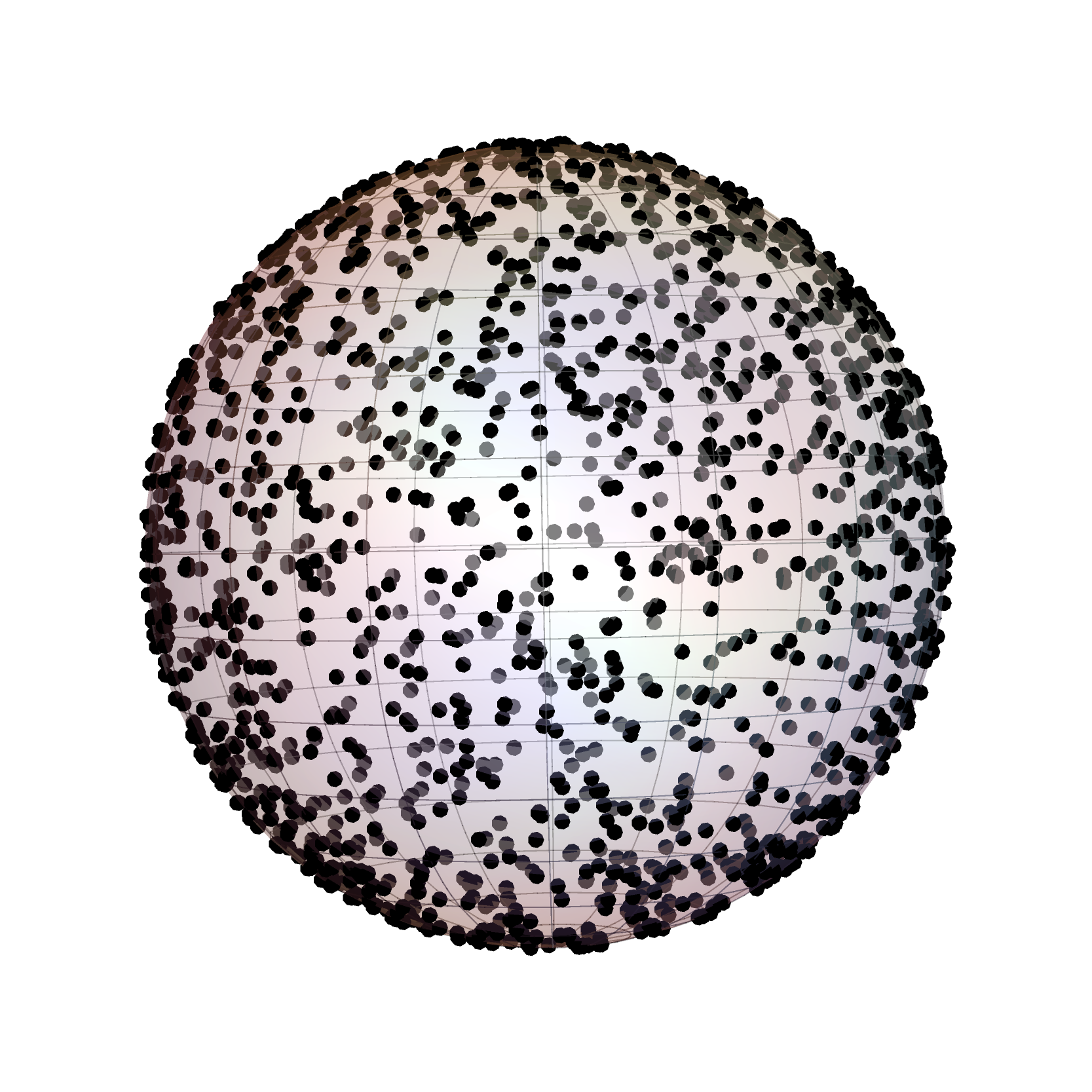}
    \end{minipage}\hfill
    \begin{minipage}[c]{0.22\textwidth}
        \caption{\label{sample3} \\ $\mu=\begin{pmatrix} 0.5 & 0 & 0 \end{pmatrix}^{\top}$ \\ $A= 0$}
    \end{minipage}
    \begin{minipage}[c]{0.27\textwidth}
        \includegraphics[width=\textwidth]{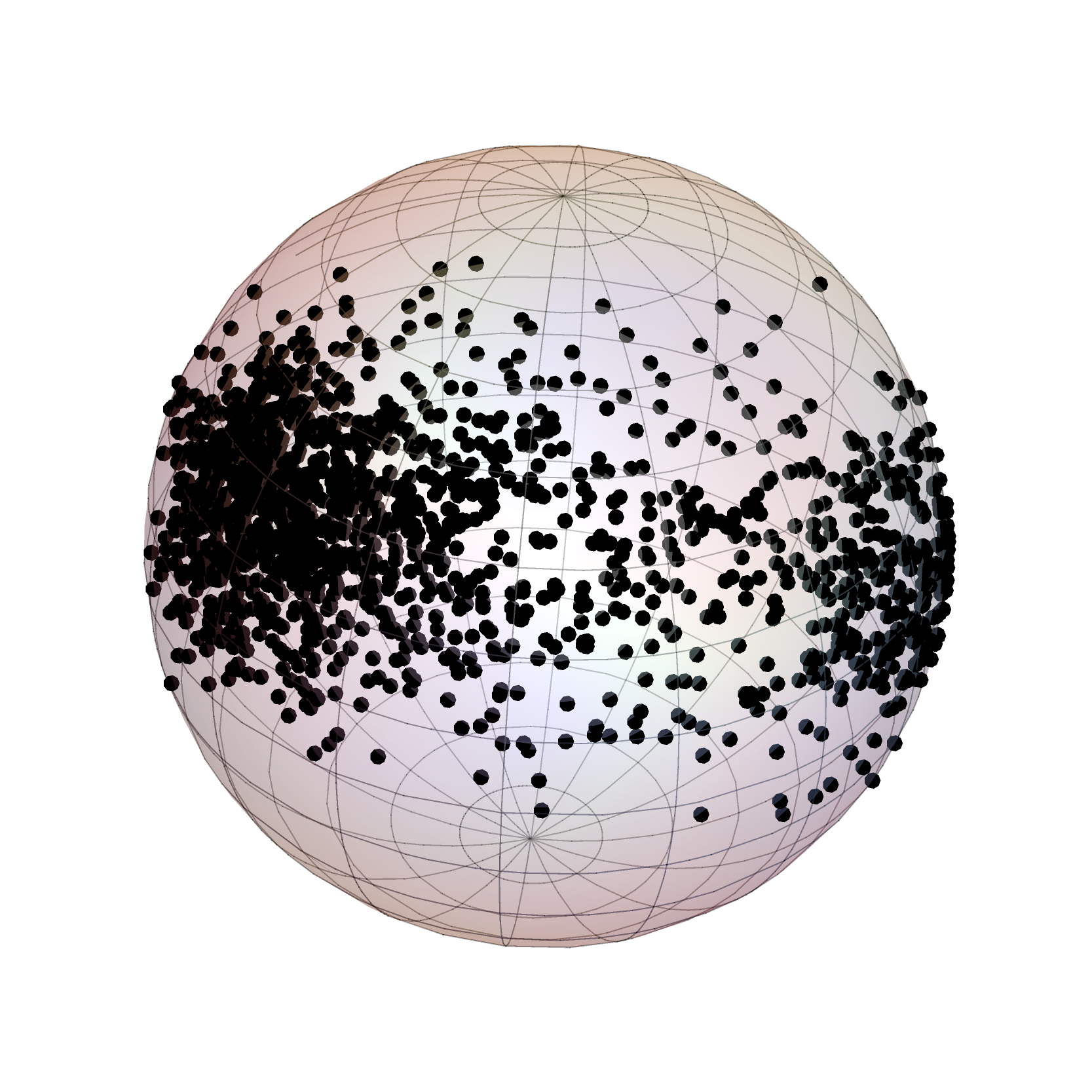}
    \end{minipage}\hfill
    \begin{minipage}[c]{0.22\textwidth}
        \caption{\label{sample4} \\ $\mu=\begin{pmatrix} 11 &3 & 10 \end{pmatrix}^{\top}$ \\ $A= \begin{pmatrix} 2 &-2 &1 \\ -2 & 12& -2 \\ 1 & -2 & 0 \end{pmatrix}$}
    \end{minipage}
    \vspace{-0.5cm}
    \begin{minipage}[c]{0.27\textwidth}
        \includegraphics[width=\textwidth]{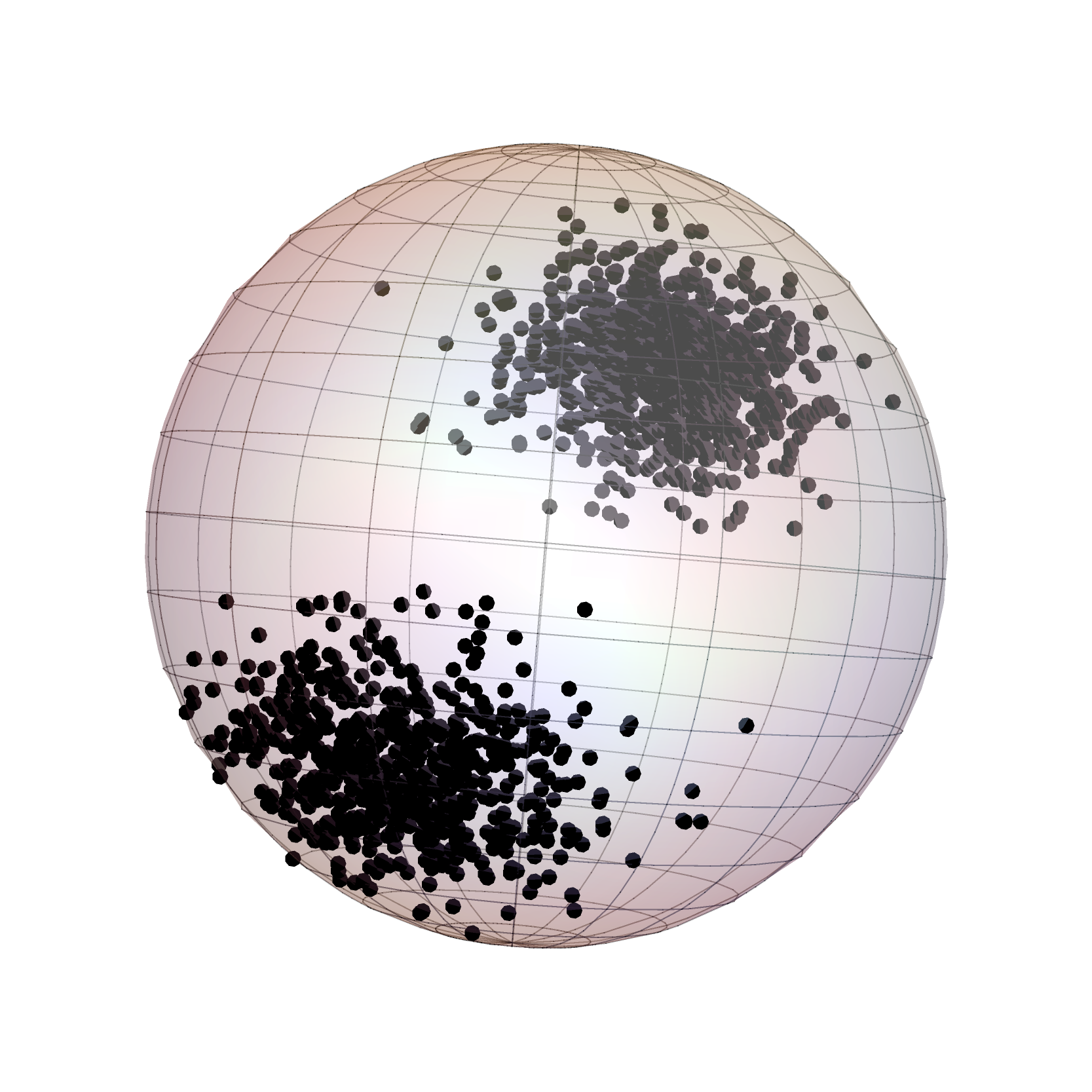}
    \end{minipage}\hfill
    \begin{minipage}[c]{0.22\textwidth}
         \caption{\label{sample5} \\ $\mu= 0.05 \cdot \begin{pmatrix} 1 & 1 & 1 \end{pmatrix}^{\top}$ \\ $A= \begin{pmatrix} 1 &2 &3 \\ 2 & 6& 7 \\ 3 & 7 & 0 \end{pmatrix}$}
    \end{minipage}
    \begin{minipage}[c]{0.27\textwidth}
        \includegraphics[width=\textwidth]{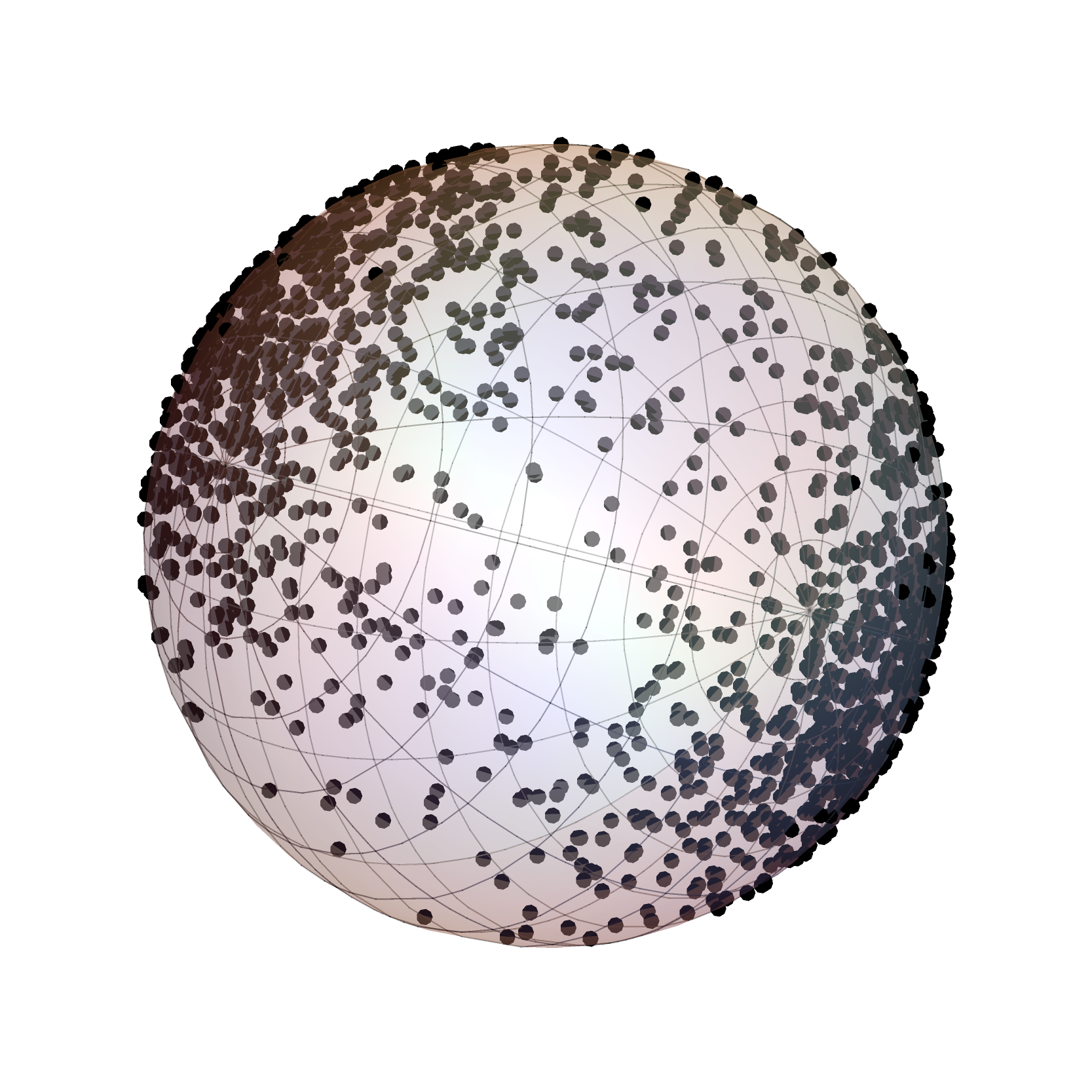}
    \end{minipage}\hfill
    \begin{minipage}[c]{0.22\textwidth}
        \caption{\label{sample6} \\ $\mu=\begin{pmatrix} 0 & 3 & 3 \end{pmatrix}^{\top}$ \\ $A= \begin{pmatrix} 0 & 0 & 0 \\ 0 & 0 & -3 \\ 0 & -3 & 0 \end{pmatrix}$}
    \end{minipage}
    \vspace{-0.5cm}
    \begin{minipage}[c]{0.27\textwidth}
        \includegraphics[width=\textwidth]{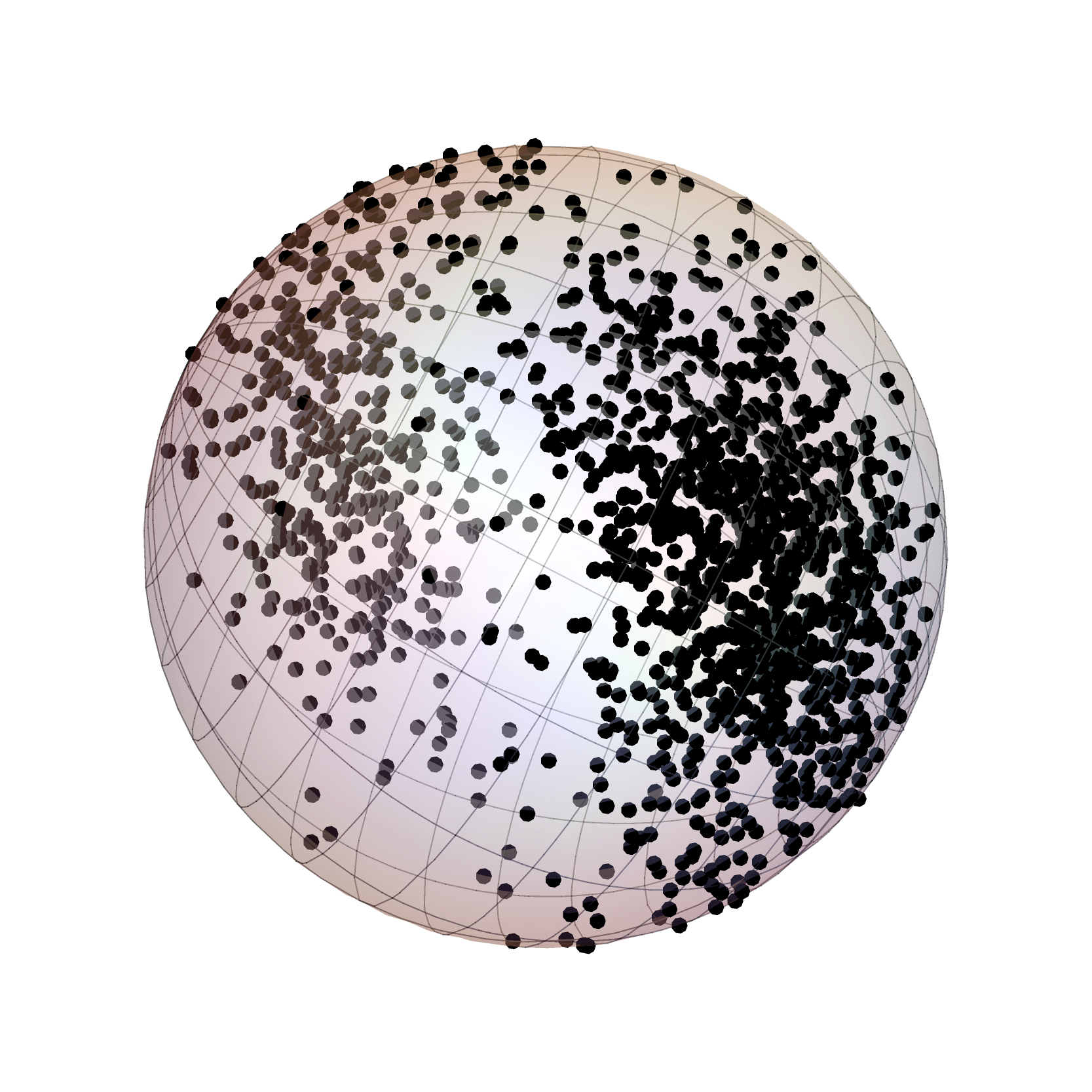}
    \end{minipage}\hfill
    \begin{minipage}[c]{0.22\textwidth}
        \caption{\label{sample7} \\ $\mu=\begin{pmatrix} 0 & 1 & 1 \end{pmatrix}^{\top}$ \\ $A= \begin{pmatrix} -1 & -2 & -3 \\ -2 & 5 & -3 \\ -3 & -3 & 0 \end{pmatrix}$}
    \end{minipage}
    \begin{minipage}[c]{0.27\textwidth}
        \includegraphics[width=\textwidth]{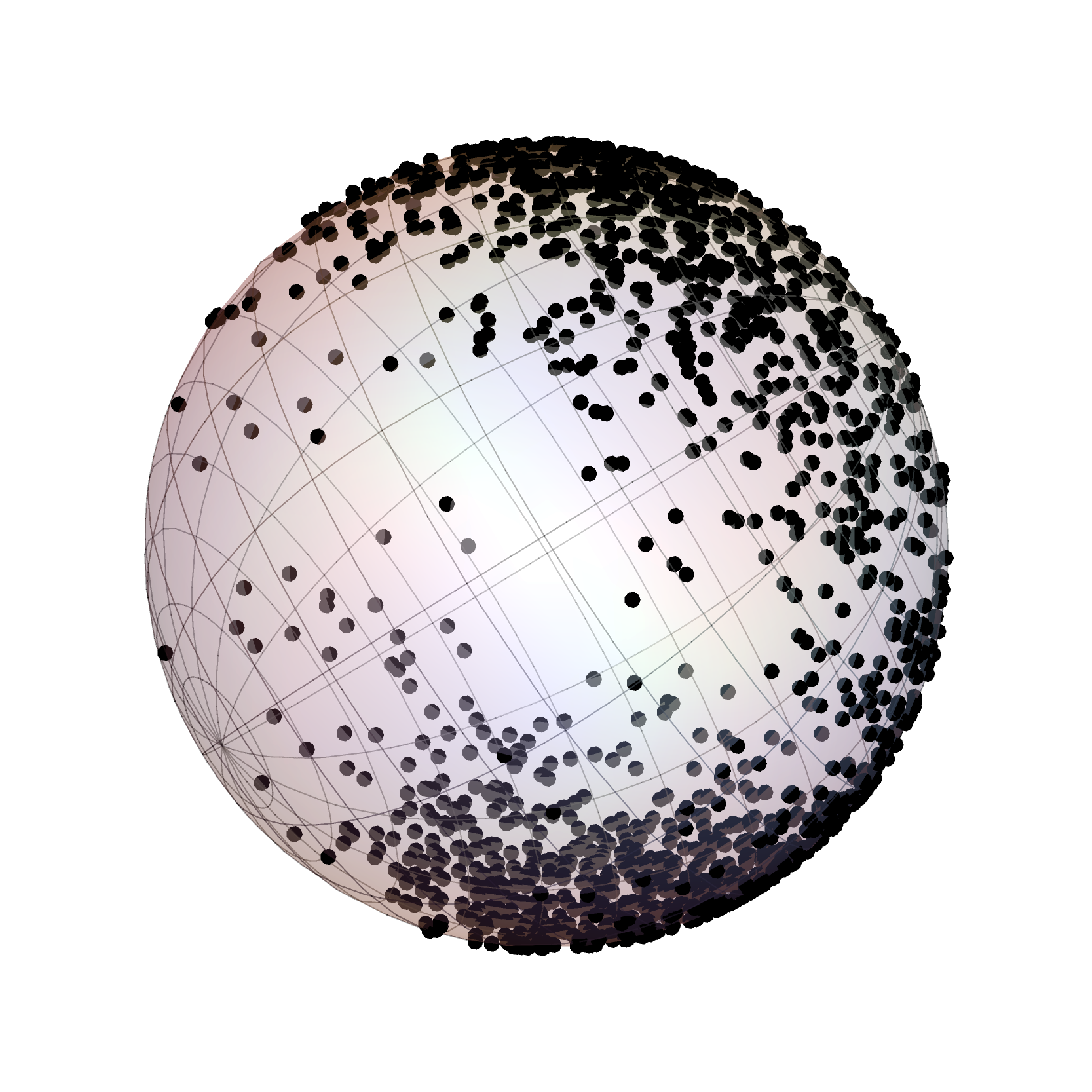}
    \end{minipage}\hfill
    \begin{minipage}[c]{0.22\textwidth}
        \caption{\label{sample8} \\ $\mu=\begin{pmatrix} 0 & -1 & 1 \end{pmatrix}^{\top}$ \\ $A= \begin{pmatrix} -1 & -2 & -3 \\ -2 & 1 &0 \\ -3 & 0 & 0 \end{pmatrix}$}
    \end{minipage}
    \vspace{-0.5cm}
    \begin{minipage}[c]{0.27\textwidth}
        \includegraphics[width=\textwidth]{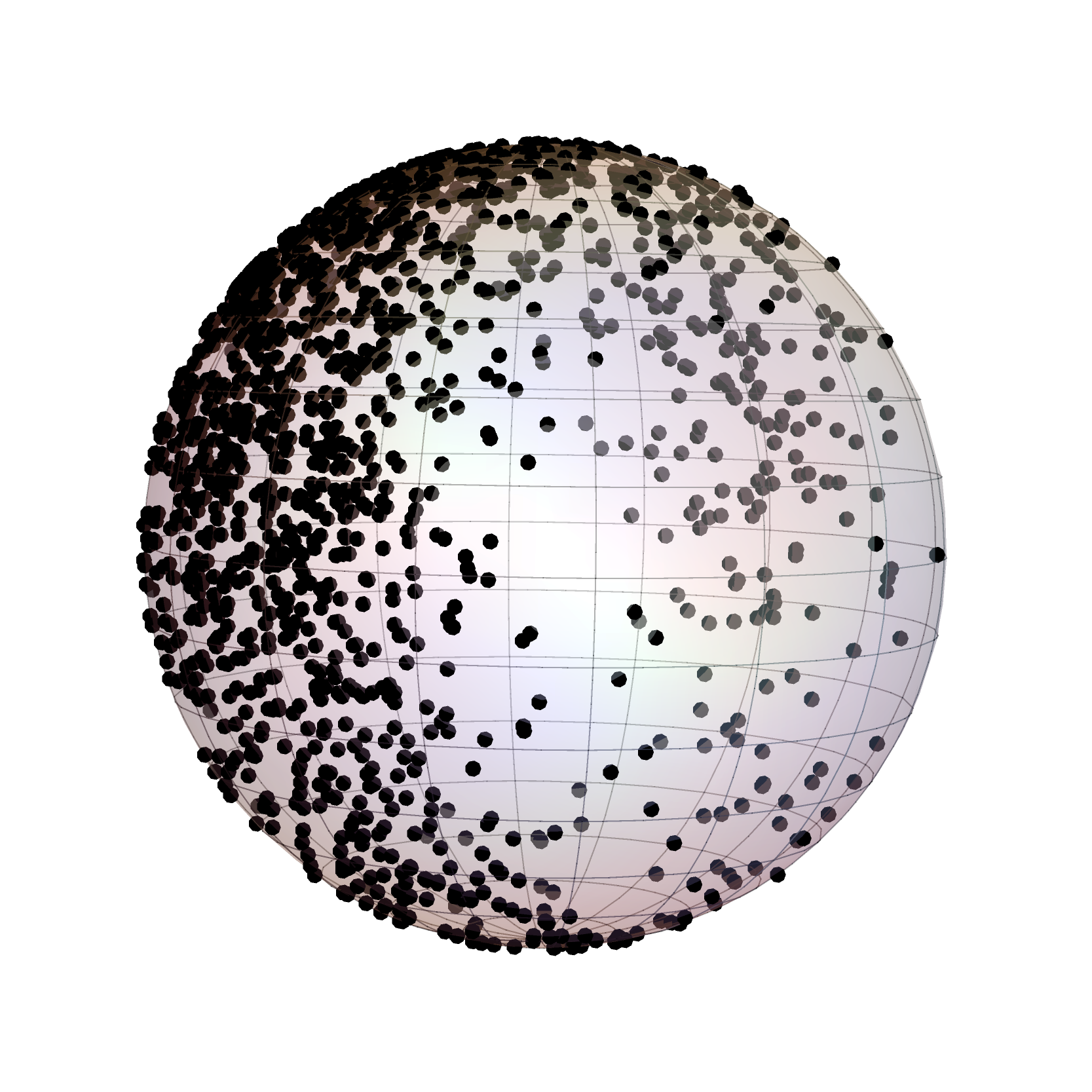}
    \end{minipage}\hfill
    \begin{minipage}[c]{0.22\textwidth}
        \caption{ \label{sample9} \\ $\mu=\begin{pmatrix} 0 &-1 & 1 \end{pmatrix}^{\top}$ \\ $A= \begin{pmatrix} -5 &0 &-1 \\ 0 & 1 & 0 \\ -1 & 0 & 0 \end{pmatrix}$}
    \end{minipage}
    \begin{minipage}[c]{0.27\textwidth}
        \includegraphics[width=\textwidth]{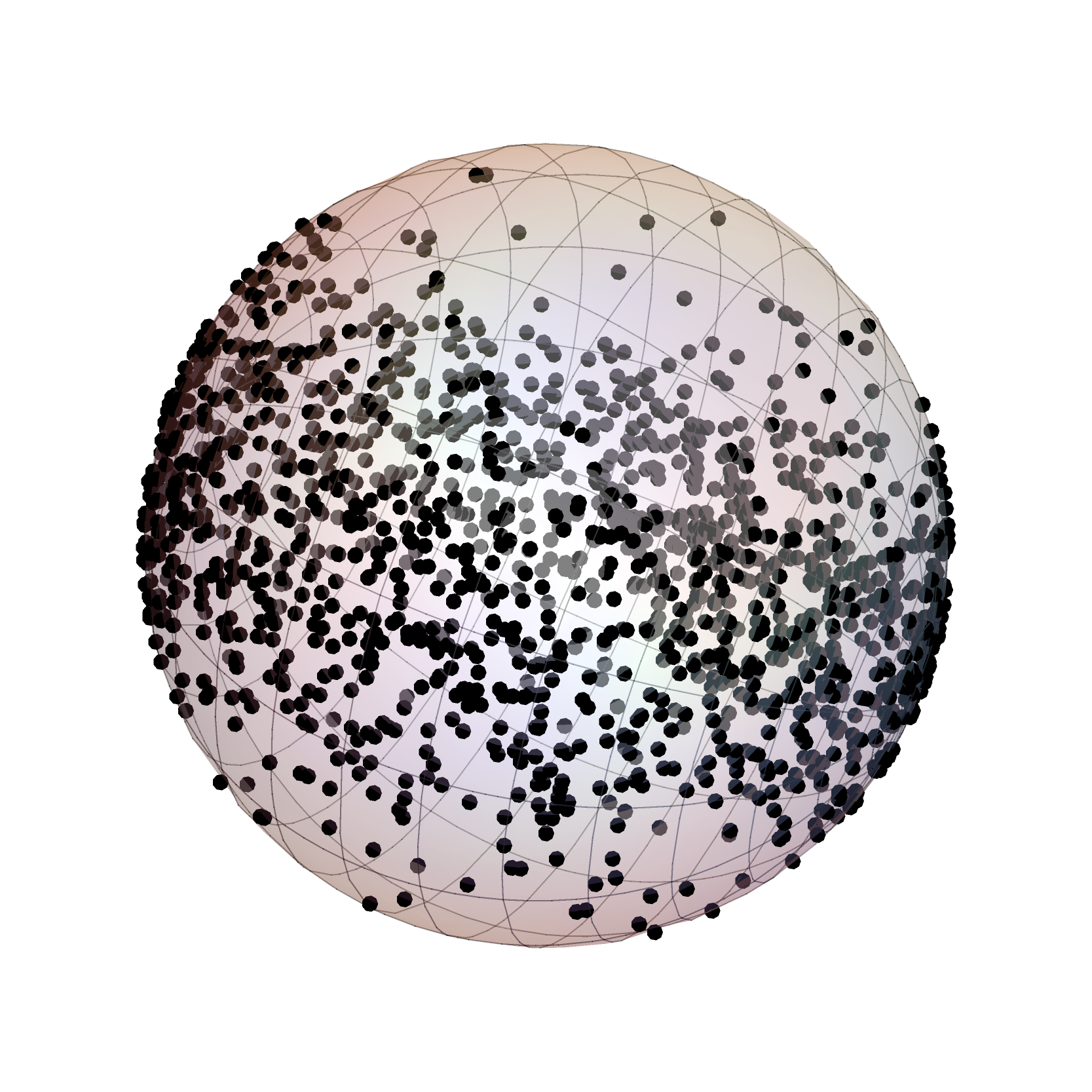}
    \end{minipage}\hfill
    \begin{minipage}[c]{0.22\textwidth}
        \caption{\label{sample10} \\ $\mu=\begin{pmatrix} 11 &3 & 10 \end{pmatrix}^{\top}$ \\ $A= \begin{pmatrix} -6 & 0 & 0 \\ 1 & 0 & 0 \\ 0 & 0 & 0 \end{pmatrix}$}
    \end{minipage}
\end{figure}
In the following, we will write $\hat{\theta}_n^{\mathrm{ST}}=(\hat{\mu}_n^{\mathrm{ST}},\hat{A}_n^{\mathrm{ST}})$ for the Stein estimator obtained with the test functions from above. \par

We briefly discuss other estimation techniques from the literature. Estimating the normalising constant $C(\mu,A)$ is difficult. However, several works have addressed the problem of computing the MLE; we refer to \cite{kume2005saddlepoint} (saddlepoint approximation), \cite{koyama2014holonomic,kume2018exact} (holonomic gradient descent) and \cite{chen2021maximum} (Euler transform). However, we did not find these methods to be natural competitors to our estimator. The saddlepoint approximation in \cite{kume2005saddlepoint} is computationally heavy; parameter estimation is only performed for the vMF distribution. In \cite{nakayama2011holonomic}, the holonomic gradient descent is developed and applied to find the MLE of two real data sets. In \cite{koyama2014holonomic}, an accelerated version of the method is proposed. However, the holonomic gradient descent also requires high computational effort; the authors in \cite{koyama2014holonomic} state that parameter estimation can be performed up to dimension $d=7$. The latter algorithm is adapted in \cite{kume2018exact} to allow for a better dimensional efficiency. \cite{chen2021maximum} implements the same algorithm as in \cite{kume2018exact} but uses an Euler transform together with a Fourier-type integral representation to compute the normalising constant and its derivatives. We implemented the algorithm in \cite{chen2021maximum} and found it to be sensitive to several initial parameters that one has to set as well as to the starting point of the MLE procedure. In addition, we refer to \cite{bee2017approximate} where an approximation of the likelihood function is used to compute the MLE for the Bingham-subfamily. \par
Due to the high computational effort and the sensitivity to initial parameters we were not able to include these estimation techniques in our simulation study. The results concerning the Stein estimator can be found in Table \ref{sim_fisherbingham}. We report the mean squared error (MSE) for both parameters, whereby we calculated for $\mu$ the average Euclidean distance and for $A$ the average spectral norm with respect to all Monte Carlo samples. We see that the Stein estimator gives reasonable results for all parameter constellations. The rather large MSE in the first column can be explained by a parameter identification problem (see Remark \ref{remark_identification_bingham}). The row NE reports the relative frequency (out of $100$) of computational failure regarding the numerical computation of $\hat{\theta}_n^{\mathrm{ST}}$. As can be seen in Table \ref{sim_fisherbingham}, there have been no issues in this respect. We used the R package \textit{simdd} \cite{simddRpackage} in order to generate samples from the Fisher-Bingham distribution.

\begin{Remark} \label{remark_identification_bingham}
    The simulation results from Table \ref{sim_fisherbingham} are conservative: In fact, two distinct distributions from the Fisher-Bingham family can be similar even though their parameter values in $\mu$ and $A$ differ largely. For instance, this seems to be the case when $\Vert \mu \Vert$ is large. Let us illustrate this problem by the following example: We generated a sample of size $n=1000$ with parameter values $\mu_0= \begin{pmatrix} 40 & 0 & 0\end{pmatrix}^{\top}$ and $A_0=0$. The Stein estimator $\hat{\theta}_n^{\mathrm{ST}}$ then gave the following estimates:
    \begin{align*}
        \hat{\mu}_n^{\mathrm{ST}}= \begin{pmatrix} 154.2 \\ 5.5 \\ -10.2  \end{pmatrix}, \quad \hat{A}_n^{\mathrm{ST}} = \begin{pmatrix} -61.3 & -2.9 & 5.4 \\ -2.9 & -1.3 & -0.8 \\ 5.4 &  -0.8  & 0 \end{pmatrix},
    \end{align*}
    which results in large values for $\Vert \hat{\mu}_n^{\mathrm{ST}} - \mu_0  \Vert $ and $\Vert \hat{A}_n^{\mathrm{ST}} - A_0  \Vert $. However, the shapes of the distributions $FB(\hat{\mu}_n^{\mathrm{ST}},\hat{A}_n^{\mathrm{ST}})$ and $FB(\mu_0,A_0)$ appear to be very similar on the sphere. This implies that the values for the MSE in Table \ref{sim_fisherbingham} are rather conservative, reporting large values for the distance between estimated and true parameter although their corresponding distributions are in truth very similar. However, we are not aware of another technique to evaluate the performance of our estimator given the complexity of the normalising constant.
\end{Remark}

\begin{table} 
\centering
\begin{tabular}{c|ccccccccccc}
 & Fig. 1 & Fig. 2 & Fig. 3 & Fig. 4 & Fig. 5 & Fig. 6 & Fig. 7 & Fig. 8 & Fig. 9 & Fig. 10 \\ \hline 
$\hat{\mu}_n^{\mathrm{ST}}$  & $2.51$ & $0.611$ & $0.092$ & $1.17$ & $0.663$ & $0.257$ & $0.294$ & $0.194$ & $0.184$ & $0.143$\\ 
 $\hat{A}_n^{\mathrm{ST}}$  & $1.65$ & $0.545$ & $0.191$ & $1.1$ & $0.835$ & $0.324$ & $0.492$ & $0.343$ & $0.347$ & $0.334$\\ 
 NE  & $0$ & $0$ & $0$ & $0$ & $0$ & $0$ & $0$ & $0$ & $0$ & $0$\\ \hline 
\end{tabular} 
\caption{\label{sim_fisherbingham} MSE for the Fisher-Bingham distribution based on $10{,}000$ Monte Carlo repetitions for a sample size $n=1000$. The different columns correspond to the parameter values in Fig. \ref{sample1} -- Fig. \ref{sample10}.}
\end{table}

\section{von Mises-Fisher distribution} \label{section_fisher_von_mises}
The pdf of the vMF distribution $vM\! F(\mu,\kappa)$, with parameter $\theta=(\mu,\kappa)$ such that $\kappa>0$ and $\mu \in \mathcal{S}^{d-1}$, is given by
\begin{align*}
    p_{\theta}(x)=\frac{\kappa^{d/2-1}}{(2\pi)^{d/2}\mathcal{I}_{d/2-1}(\kappa)} \exp\big(\kappa \mu^{\top} x\big), \quad x \in \mathcal{S}^{d-1},
\end{align*}
where $\mathcal{I}_{\nu}(\cdot)$ denotes the modified Bessel function of the first kind of order $\nu$ (we have used calligraphic $\mathcal{I}$ here in our notation to avoid confusion with our notation for the $d\times d$ identity matrix $I_d$). The vMF distribution is a subclass of the Fisher-Bingham distribution and is retrieved from the latter by choosing $A=0$. Let $X_1,\ldots,X_n \sim vM\! F(\mu_0,\kappa_0) $ be an i.i.d.\ sample defined on a common probability space $(\Omega,\mathcal{F},\mathbb{P})$. A standard estimator for $\mu_0$ is the directional sample mean $\hat{\mu}_n=\overline{X}/\Vert \overline{X} \Vert $ (see \cite[Chapter 10.3]{mardia2000directional}). As per $\kappa$, we develop an estimator based on the Stein identity \eqref{Stein_identity_manifolds}. We therefore choose a test function $f:\mathcal{S}^{d-1} \rightarrow \mathbb{R}^d$ for which each component is an element of $f \in \mathscr{F}$ and consider the equation
\begin{align*}
    \frac{1}{n} \sum_{i=1}^n \mathcal{A}_{\theta}f(X_i) =0.
\end{align*}
Note that in the display above we find $d$ equations although we have a scalar parameter $\kappa$ to estimate. Let $\nabla f$ and $\nabla^2 f$ be defined as in Section \ref{section_fisherbingham}. With
\begin{alignat*}{2}
    Q_n&= \overline{ (d-1) \nabla f(X) X + \nabla^2 f(X) (X \otimes X) - \Delta f(X) } \quad && \in \mathbb{R}^d,  \\
    K_n&=\overline{\nabla f(X) (I_d-XX^{\top})}\hat{\mu}_n \quad && \in \mathbb{R}^d
\end{alignat*}
the estimator for $\kappa$ is the least-squares type estimator
\begin{align*}
    \hat{\kappa}_n=(K_n^{\top}K_n)^{-1}K_n^{\top}Q_n.
\end{align*}
Conditions for consistency and existence can be worked out completely similarly to Theorem \ref{theoem_consistency_fisherbingham}. We found the test function $f(x)=x$ a convenient choice, which leads to
\begin{align} \label{estimator_fisher_von_mises}
    \hat{\kappa}_n=\frac{(d-1)\hat{\mu}_n^{\top}(I_d-\overline{XX^{\top}})\overline{X}}{\hat{\mu}_n^{\top}(I_d-\overline{XX^{\top}})^2\hat{\mu}_n}.
\end{align}
We will write $\hat{\kappa}_n^{\mathrm{ST}}$ for the estimator \eqref{estimator_fisher_von_mises}. It can be shown easily that $\hat{\kappa}_n^{\mathrm{ST}}$ is always greater than zero and independent of $\mu_0$. \par

The asymptotic distribution of the directional sample mean has been investigated in \cite{hendriks1996asymptotic} (see also \cite{hendriks1996asymptotica}). In the following theorem we give the asymptotic covariance for the estimator in \eqref{estimator_fisher_von_mises}.
\begin{Theorem} \label{theorem_asympt_cov_fishervonmises}
    The Stein estimator $\hat{\kappa}_n^{\mathrm{ST}}$ is asymptotically normal, i.e.
    \begin{align*}
        \sqrt{n} \big(  \hat{\kappa}_n^{\mathrm{ST}}-\kappa_0 \big)
        \overset{D}{\longrightarrow} N(0,P),
    \end{align*}
    as $n \rightarrow \infty$, with
    \begin{align*}
        P= \frac{\kappa_0 \mathcal{I}_{d/2-1}(\kappa_0) \big( 2 \kappa_0 \mathcal{I}_{d/2-1}(\kappa_0) - (d+1) \mathcal{I}_{d/2}(\kappa_0)  \big)} {(d-1) \mathcal{I}_{d/2}(\kappa_0)^2}.
    \end{align*}
\end{Theorem}
We now compare our new estimator $\hat{\kappa}_n^{\mathrm{ST}}$ to the MLE $\hat{\kappa}_n^{\mathrm{ML}}$, which is the solution to
\begin{align*}
    \frac{\mathcal{I}_{d/2}(\hat{\kappa}_n^{\mathrm{ML}})}{\mathcal{I}_{d/2-1}(\hat{\kappa}_n^{\mathrm{ML}})}= \Vert \overline{X}
     \Vert.
\end{align*}
The MLE is not explicit and requires numerical methods (we refer to \cite{ley2017modern} for an overview of numerical approximations for the MLE). Simulation results can be found in Table \ref{fishervonmises_sim} for different values of $\kappa_0$ and dimensions of the sphere. We also included the score-matching estimator $\hat{\kappa}_n^{\mathrm{SM}}$ from \cite{mardia2016score}, which is defined as $\hat{\kappa}_n^{\mathrm{SM}}=(d-1)\overline{Y}/(1-\overline{Y^2})$, where $Y_i$ is the first component of the vector $RX_i, \, i=1,\ldots,n$, with $R \in \mathbb{R}^{d \times d}$ orthogonal such that $R\hat{\mu}_n=(1,0,\ldots,0)^{\top}$. We observe that our proposed estimator $\hat{\kappa}_n^{\mathrm{ST}}$ performs well in terms of bias and MSE, outperforming (or equal in performance) the score matching estimator throughout all parameter values and the MLE for most parameter values. For smaller sample sizes we also found $\hat{\kappa}_n^{\mathrm{ST}}$ to be more reliable in comparison to $\hat{\kappa}_n^{\mathrm{ML}}$ and $\hat{\kappa}_n^{\mathrm{SM}}$. The parameter $\mu$ was assumed to be not known and was estimated by the directional sample mean.

\begin{Remark}
The Fisher information with respect to the parameter $\kappa$ of the $vM\! F(\mu, \kappa)$ distribution is given by
\begin{align*}
    \mathscr{I}(\kappa)=1-\frac{\mathcal{I}_{d/2}(\kappa)^2}{\mathcal{I}_{d/2-1}(\kappa)^2} - \frac{(d-1)\mathcal{I}_{d/2}(\kappa)}{\kappa\mathcal{I}_{d/2-1}(\kappa)}.
\end{align*}
It is well-known that the asymptotic variance of $\hat{\kappa}_n^{\mathrm{ML}}$ is given by $1/\mathscr{I}(\kappa_0)$. A plot of the asymptotic variances with respect to $\kappa_0$ of both estimators, $\hat{\kappa}_n^{\mathrm{ML}}$ and $\hat{\kappa}_n^{\mathrm{ST}}$, shows that they are very close with the MLE obviously having the smaller variance since it is asymptotically efficient. In the case $d=2$, the asymptotic variance of the Stein estimator equals the one of $\hat{\kappa}_n^{\mathrm{SM}}$ which is computed in \cite{mardia2016score}. However, the authors do not give a formula for higher dimensions.
\end{Remark}

\begin{Remark}
    We briefly demonstrate another possibility to obtain an estimator for the concentration parameter: Write $\mu'=\kappa \mu$ and solve the  Stein identity for $\mu'$ by considering the test function $f:\mathcal{S}^{d-1} \rightarrow \mathbb{R}^d, x \mapsto x$. The solution is given by $\hat{\mu}'_n=(d-1) (I_d - \overline{XX^{\top}})^{-1} \overline{X}$. Then define the estimator
    \begin{align*}
        \hat{\kappa}_n^{\mathrm{ST2}}= (d-1) \Vert (I_d - \overline{XX^{\top}})^{-1} \overline{X} \Vert.
    \end{align*}
    This estimator seems to behave very similar to $\hat{\kappa}_n^{\mathrm{ST}}$ for both, small and large sample sizes.
\end{Remark}

\begin{table} 
\centering
\begin{tabular}{cc|ccc|ccc}
 $\theta_0$ & & \multicolumn{3}{|c}{Bias} & \multicolumn{3}{|c}{MSE} \\ \hline
 & & $\hat{\theta}_n^{\mathrm{ML}}$ & $\hat{\theta}_n^{\mathrm{SM}}$ &  $\hat{\theta}_n^{\mathrm{ST}}$ & $\hat{\theta}_n^{\mathrm{ML}}$ & $\hat{\theta}_n^{\mathrm{SM}}$ & $\hat{\theta}_n^{\mathrm{ST}}$  \\ \hline
\multirow{1}{*}{$d=3,\kappa_0=1$} & $\kappa$  & $0.044$ & $0.046$ & \hl $\hl 0.043$ & \hl $\hl 0.039$ & $0.041$ & $0.041$ \\
\multirow{1}{*}{$d=3,\kappa_0=2$} & $\kappa$  & $0.046$ & $0.048$ & \hl $\hl 0.043$ & \hl $\hl 0.063$ & $0.07$ & $0.069$ \\
\multirow{1}{*}{$d=3,\kappa_0=10$} & $\kappa$  & $0.219$ & $0.211$ & \hl $\hl 0.201$ & \hl $\hl 1.14$ & $1.15$ & $1.15$ \\
\multirow{1}{*}{$d=3,\kappa_0=50$} & $\kappa$   & $0.938$ & $0.929$ & \hl $\hl 0.919$ & \hl $\hl 28.1$ & \hl $\hl 28.1$ & \hl $\hl 28.1$ \\
\multirow{1}{*}{$d=10,\kappa_0=1$} & $\kappa$ & $0.402$ & $0.402$ & \hl $\hl 0.4$ & $0.241$ & $0.241$ & \hl $\hl 0.24$ \\
\multirow{1}{*}{$d=10,\kappa_0=10$} & $\kappa$  & $0.183$ & $0.179$ & \hl $\hl 0.173$ & \hl $\hl 0.412$ & $0.414$ & \hl $\hl 0.412$ \\
\multirow{1}{*}{$d=10,\kappa_0=50$} & $\kappa$& $0.696$ & $0.69$ & \hl $\hl 0.684$ & $6.71$ & $6.7$ & \hl $\hl 6.69$ \\
\multirow{1}{*}{$d=20,\kappa_0=1$} & $\kappa$  & \hl $\hl 1.23$ & \hl $\hl 1.23$ & \hl $\hl 1.23$ & \hl $\hl 1.64$ & \hl $\hl 1.64$ & \hl $\hl 1.64$ \\
\multirow{1}{*}{$d=20,\kappa_0=10$} & $\kappa$   & $0.327$ & $0.326$ & \hl $\hl 0.323$ & \hl $\hl 0.454$ & $0.457$ & \hl $\hl 0.454$ \\
\multirow{1}{*}{$d=20,\kappa_0=50$} & $\kappa$ & $0.666$ & $0.661$ & \hl $\hl 0.656$ & \hl $\hl 3.71$ & \hl $\hl 3.71$ & \hl $\hl 3.71$ \\  \hline
\end{tabular} 
\caption{\label{fishervonmises_sim} Simulation results for the $vM\! F(\mu,\kappa)$ distribution for $\mu_0$ proportional to $(1,\ldots,1)$, sample size $n=100$, and $10{,}000$ repetitions.}
\end{table}
\section{Watson distribution} \label{section_watson}
The pdf of the Watson distribution $W(\mu,\kappa)$, with parameter $\theta=(\mu,\kappa)$ such that $\kappa \in \mathbb{R}$ and $\mu \in \mathcal{S}^{d-1}$, is given by
\begin{align*}
    p_{\theta}(x)=\frac{\Gamma(d/2)}{2\pi^{d/2} \, _1F_1(1/2;d/2;\kappa)} \exp\big(\kappa (\mu^{\top} x)^2\big), \quad x \in \mathcal{S}^{d-1}.
\end{align*}
If $\kappa=0$, the parameter $\mu$ is no longer identifiable and we obtain the uniform distribution on the sphere. The Watson distribution is a subfamily of the Fisher-Bingham distribution if we allow the concentration parameter to be equal to $0$. Let $\mu^{\mathrm{FB}}$, $A^{\mathrm{FB}}$ be the parameters in the Fisher-Bingham parametrization. Then the Watson distribution with parameters $\kappa$, $\mu$ as above is obtained by choosing $\mu^{\mathrm{FB}}=0$ and $A^{\mathrm{FB}}=\kappa \mu \mu^{\top}$. \par

Suppose now that we are given an i.i.d.\ sample $X_1,\ldots,X_n \sim W(\mu_0, \kappa_0)$ defined on a common probability space $(\Omega,\mathcal{F},\mathbb{P})$. \par

A suitable estimator for $\mu_0$ is the MLE which is given by the largest resp.\ smallest eigenvector of the scatter matrix $S_n=n^{-1}\sum_{i=1}^n X_i X_i^{\top}$ if $\kappa_0>0$ resp.\ $\kappa_0<0$ (see, for example, \cite{sra2013multivariate}, we also refer to \cite{dey2022inference}, in which the authors develop a new class of equivariant estimators that yield restricted MLE and a Bayesian estimator under some non-informative prior). We take the latter estimator and write $\hat{\mu}_n^{\bullet}$ with $\bullet= (+)$ for the largest and $\bullet= (-)$ for the smallest eigenvector of $S_n$. Given $\hat{\mu}_n^{\bullet}$ we now propose an estimator $\hat{\kappa}_n$ based on the Stein operator.

Let $f:\mathcal{S}^{d-1} \rightarrow \mathbb{R}^{d(d+1)/2-1}$ be a test function with all components in the class $\mathscr{F}$, where again $\nabla f$ and $\nabla^2 f$ are defined as in Section \ref{section_fisherbingham}. Given $\hat{\mu}_n^{\bullet}$, we define the quantities
\begin{alignat*}{2}
    V_n&= \overline{ (d-1) \nabla f(X) X + \nabla^2 f(X) (X \otimes X) - \Delta f(X) } \quad && \in \mathbb{R}^{d(d+1)/2-1 }, \\
    J_n^{\bullet}&=2 \overline{\nabla f(X) (I_d-XX^{\top})\hat{\mu}_n^{\bullet}(\hat{\mu}_n^{\bullet})^{\top}X} \quad && \in \mathbb{R}^{d(d+1)/2-1 }.
\end{alignat*}
We then propose a least-squares type estimator by
\begin{align*}
    \hat{\kappa}_n^{\bullet}=((J_n^{\bullet})^{\top}J_n^{\bullet})^{-1}(J_n^{\bullet})^{\top}V_n.
\end{align*}
The Stein estimator $\hat{\theta}_n=(\hat{\mu}_n,\hat{\kappa}_n)$ is now defined as follows: Calculate ($\hat{\mu}_n^{(-)},\hat{\kappa}_n^{(-)})$ as well as $(\hat{\mu}_n^{(+)},\hat{\kappa}_n^{(+)})$. If neither of the two estimates is eligible ($\hat{\kappa}_n^{(-)}>0$ and $\hat{\kappa}_n^{(+)}<0$), $\hat{\theta}_n$ does not exist; if exactly one of the estimators is eligible, we choose the latter, and if both are eligible, we choose the one for which $\Vert J_n^{\bullet}\hat{\kappa}_n^{\bullet} - V_n \Vert $ is smaller. Conditions for consistency and existence can be worked out completely similarly to Theorem \ref{theoem_consistency_fisherbingham}. We only show that this estimator is asymptotically normal and do not work out the covariance matrix given its complicated structure. We let
\begin{align*}
    Y_n=\begin{pmatrix} V_n \\  \overline{X^{\top} \otimes \big(\nabla f_2(X) \big(I_d-XX^{\top}\big) \big)} \\[3pt] \mathrm{vec}\big(\overline{XX^{\top}} \big) \end{pmatrix}.
\end{align*}
\begin{Theorem} \label{theorem_watson_asyp_norm}
    Let $X \sim W(\mu_0,\kappa_0)$. Suppose that $\mathrm{Var}[Y_1]$ exists and is invertible. Moreover, suppose that $ J^{\top}J \neq 0$, where
    \begin{align*}
        J=\mathbb{E}\big[\nabla f_2(X) (I_d-XX^{\top})\mu_0 \mu_0^{\top}X\big],
    \end{align*}
   and assume that the two largest resp.\ smallest eigenvalues of $\mathbb{E}[XX^{\top}]$ are distinct for $\kappa_0\geq 0$ resp.\ $\kappa_0\leq 0$. Then the sequence  $ \sqrt{n} (  \hat{\kappa}_n-\kappa_0) $ converges in distribution to a mean zero normal distribution as $n \rightarrow \infty$.
\end{Theorem}

Notice that the assumptions made in Theorem \ref{theorem_watson_asyp_norm} entail that $\kappa_0 \neq 0$ since we have $\mathbb{E}[XX^{\top}]=d^{-1} I_d$ for $\kappa_0=0$ and the eigenvalues cannot be distinguished.  \par
Motivated by the estimator of the matrix $A$ in the case of the Fisher-Bingham distribution, we choose the test function $f_2(x)=\mathrm{vech}'(xx^{\top})$. In the following, we will write $\hat{\kappa}_n^{\mathrm{ST}}$ for the latter choice of the test function. We also consider two other estimators. First, in \cite{sra2013multivariate} the authors develop an explicit approximation of the MLE. In fact, they show that, given the correct sign of $\kappa_0$, the MLE $\hat{\kappa}_n^{\mathrm{ML}}$ always lies between 
\begin{align*}
    L\big(\hat{\mu}_n^{\top} S \hat{\mu}_n, 1/2,d/2\big) < \hat{\kappa}_n^{\mathrm{ML}} < U\big(\hat{\mu}_n^{\top} S \hat{\mu}_n, 1/2,d/2\big),
\end{align*}
where 
\begin{align*}
    L(r,a,c)=\frac{rc-a}{r(1-r)} \bigg( 1+ \frac{1-r}{c-a} \bigg) 
 ,\quad
    U(r,a,c)=\frac{rc-a}{r(1-r)} \bigg( 1+ \frac{r}{a} \bigg).
\end{align*}
The bounds are asymptotically sharp. Therefore, we include the estimator defined by
\begin{align*}
    (\hat{\kappa}_n^{\mathrm{MLa}})^{\bullet}= \frac{1}{2} \Big(L\big((\hat{\mu}_n^{\bullet})^{\top} S \hat{\mu}_n^{\bullet}, 1/2,d/2\big) + U\big((\hat{\mu}_n^{\bullet})^{\top} S \hat{\mu}_n^{\bullet}, 1/2,d/2\big)  \Big)
\end{align*}
in our simulation study. The question of how to choose the correct eigenvector of the scatter matrix is tackled the same way as for $\hat{\theta}_n^{\mathrm{ST}}$ except: If $( \hat{\mu}_n^{(-)}, (\hat{\kappa}_n^{\mathrm{MLa}})^{(-)})$ and $(\hat{\mu}_n^{(+)}, (\hat{\kappa}_n^{\mathrm{MLa}})^{(+)})$ are eligible, we chose the one with the higher likelihood. We denote the resulting estimator by $\hat{\theta}_n^{\mathrm{MLa}}$. \par
Next, we included the MLE via the \textit{watson} R package \cite{watsonRpackage} which implements an EM algorithm that also works for finite mixtures of the Watson distribution. This estimator is denoted by $\hat{\theta}_n^{\mathrm{EM}}$. 

Simulation results are given in Table \ref{watson_sim} for a variety of dimensions and parameter values for $\kappa_0$. We only report the results for $\kappa$ since the methods do not differ in estimating $\mu$. However, we also included a column NE reporting as before the estimated relative frequency that the estimator does not exist (which means in this case that the value of $\mu$ and the sign of $\kappa$ do not match). As can be observed from Table \ref{watson_sim}, the EM estimator $\hat{\theta}_n^{\mathrm{EM}}$ seems to be globally the best in terms of MSE, whilst the Stein estimators performs globally best in terms of bias. The Stein estimator also outperforms $\hat{\theta}_n^{\mathrm{MLa}}$ for most parameter constellations in terms of MSE.

\begin{table} 
\centering
\begin{tabular}{cc|ccc|ccc|ccc}
 $\theta_0$ & & \multicolumn{3}{|c}{Bias} & \multicolumn{3}{|c}{MSE} & \multicolumn{3}{|c}{NE} \\ \hline
 & & $\hat{\theta}_n^{\mathrm{MLa}}$ &  $\hat{\theta}_n^{\mathrm{EM}}$ & $\hat{\theta}_n^{\mathrm{ST}}$ & $\hat{\theta}_n^{\mathrm{MLa}}$ & $\hat{\theta}_n^{\mathrm{EM}}$ & $\hat{\theta}_n^{\mathrm{ST}}$ & $\hat{\theta}_n^{\mathrm{MLa}}$ & $\hat{\theta}_n^{\mathrm{EM}}$ & $\hat{\theta}_n^{\mathrm{ST}}$ \\ \hline
\multirow{1}{*}{$d=3,\kappa_0=-20$} & $\kappa$  & $-9.7$ & $-0.834$ & \hl $\hl -0.792$ & $114$ & \hl $\hl 9.82$ & $9.85$ & \multirow{1}{*}{$1$} & \multirow{1}{*}{$0$} & \multirow{1}{*}{$0$} \\  
\multirow{1}{*}{$d=3,\kappa_0=-10$} & $\kappa$  & $-4.29$ & $-0.423$ & \hl $\hl -0.384$ & $23.6$ & \hl $\hl 2.49$ & $2.52$ & \multirow{1}{*}{$0$} & \multirow{1}{*}{$0$} & \multirow{1}{*}{$0$} \\
\multirow{1}{*}{$d=10,\kappa_0=-10$} & $\kappa$  & $-3.28$ & \hl $\hl -2.43$ & $-3.26$ & $18.4$ & \hl $\hl 12.2$ & $22.1$ & \multirow{1}{*}{$0$} & \multirow{1}{*}{$0$} & \multirow{1}{*}{$0$} \\ 
\multirow{1}{*}{$d=10,\kappa_0=-2$} & $\kappa$  & $-1.27$ & $-1.1$ & \hl $\hl -0.831$ & $16.7$ & \hl $\hl 15.6$ & $18.3$ & \multirow{1}{*}{$0$} & \multirow{1}{*}{$0$} & \multirow{1}{*}{$0$} \\ 
\multirow{1}{*}{$d=20,\kappa_0=-2$} & $\kappa$  & $-15.3$ & $-14.7$ & \hl $\hl -12.7$ & $309$ & \hl $\hl 288$ & $356$ & \multirow{1}{*}{$0$} & \multirow{1}{*}{$0$} & \multirow{1}{*}{$0$} \\ 
\multirow{1}{*}{$d=3,\kappa_0=1$} & $\kappa$  & \hl $\hl -0.262$ & $-0.269$ & $-0.303$ & $1$ & \hl $\hl 0.955$ & $0.994$ & \multirow{1}{*}{$0$} & \multirow{1}{*}{$0$} & \multirow{1}{*}{$0$} \\ 
\multirow{1}{*}{$d=10,\kappa_0=1$} & $\kappa$  & $-3.48$ & $-3.34$ & \hl $\hl -3.22$ & $27.7$ & \hl $\hl 26.1$ & $27.4$ & \multirow{1}{*}{$0$} & \multirow{1}{*}{$0$} & \multirow{1}{*}{$0$} \\ 
\multirow{1}{*}{$d=20,\kappa_0=5$} & $\kappa$  & $-13.7$ & $-13.3$ & \hl $\hl -9.01$ & $393$ & $373$ & \hl $\hl 297$ & \multirow{1}{*}{$0$} & \multirow{1}{*}{$0$} & \multirow{1}{*}{$0$} \\ 
\multirow{1}{*}{$d=3,\kappa_0=10$} & $\kappa$  & $7.44$ & $0.188$ & \hl $\hl 0.067$ & $59.1$ & \hl $\hl 0.927$ & $0.955$ & \multirow{1}{*}{$0$} & \multirow{1}{*}{$0$} & \multirow{1}{*}{$0$} \\ 
\multirow{1}{*}{$d=10,\kappa_0=20$} & $\kappa$  & $13.8$ & $0.216$ & \hl $\hl -0.087$ & $194$ & $0.881$ & \hl $\hl 0.876$ & \multirow{1}{*}{$0$} & \multirow{1}{*}{$0$} & \multirow{1}{*}{$0$} \\ \hline
\end{tabular}
\caption{\label{watson_sim} Simulation results for the $W(\mu,\kappa)$ distribution for $\mu_0$ proportional to $(1,\ldots,1)$ with sample size $n=100$, and $10{,}000$ repetitions.}
\end{table}

\section{Real data example} \label{section_real_data_example}
We conclude the paper with a real data example from machine learning. An \textit{autoencoder} consists of two artificial neural networks, an encoder and a decoder that are connected in series (we refer to \cite[Chapter 14]{goodfellow2016deep} for a comprehensive introduction). The encoder converts the input data into a representation on a latent space and the decoder reconstructs the original input from the latent space. Autoencoders are, for example, used for dimension reduction or regularization. Meanwhile, there are many modifications of autoencoders, for example the \textit{variational autoencoder} where the input is not mapped directly onto the latent space: Instead, the final layer of the encoder represents the parametrization of a probability distribution on the latent space and then a random variable sampled according to the parameter values from the encoder output is used to obtain a point on the latent space.

Variational autoencoders can also be seen as a variational Bayesian method in which the decoder computes the likelihood distribution of the dataset (conditional on the latent space representation) and the encoder computes the approximated posterior distribution (the distribution on the latent space given the data). Moreover, due to their probabilistic nature, variational autoencoders are generative models. We refer to \cite{cinelli2021variational,kingma2019introduction} for an introduction to the concept of variational autoencoders. For most applications, the Euclidean space together with a multivariate normal distribution $N(\mu, \Sigma)$ is used to model the distribution of the encoded data. However, in \cite{davidson2018hyperspherical} the authors suggest to use hyperspherical latent spaces together with the vMF distribution as it allows for a truly uninformative prior. Based on their approach, \cite{davidson2018hyperspherical} used the Fisher-Bingham distribution to estimate the latent space representations of the MNIST dataset and then applied the decoder to sampled random variables from the estimated distributions in order to generate handwritten numbers. In \cite{zhao2019latent}, the authors remove the probabilistic part of the model and use a purely deterministic autoencoder to obtain latent space representations. 

In our approach, we combine the latter two methods with some additional modifications: We use a purely deterministic autoencoder to obtain latent space representations and learn the spherical encodings for each digit separately instead of using the same autoencoder for the whole dataset. Next, we employ the new estimator from Section \ref{section_fisherbingham} to fit the Fisher-Bingham distribution to each latent representation and then sample from the estimated distributions to generate handwritten digits via the decoder.  \par

The MNIST dataset set consists of labeled (from $0$ to $9$) images with handwritten numbers ($60{,}000$ training and $10{,}000$ test images) where each image has a dimension of $28 \times 28$ pixels and each pixel corresponds to a value between $0$ (black) and $1$ (white). As already mentioned above, we divided the training as well as the test set into different groups corresponding to the label of the image and then trained an autoencoder separately for each of the $10$ groups. However, the setup for each autoencoder was the same: For the encoder, we used next to the input layer of $784$ nodes two hidden layers ($256$ and $128$ nodes) and an output layer of $6$ nodes. The decoder was implemented with the same (reversed) structure. We used the ReLU activation function for all layers, except for the last layer of the encoder and decoder, respectively. Regarding the output layer of the decoder, we implemented the sigmoid function and for the output layer of the encoder, we employed the function
\begin{align*}
    x \mapsto \frac{x - \overline{x}}{\Vert x - \overline{x} \Vert},
\end{align*}
which was proposed by \cite{zhao2019latent}. This transformation yields a unit vector and therefore the latent representation of the handwritten digits. Note that through the centralization we loose one dimension: If the encoder output layer consists of $d$ nodes, then we obtain latent space representations in $\mathcal{S}^{d-2}$. Each of the $10$ models was trained for $1500$ epochs and then we applied early stopping with a look-ahead of $50$ epochs. Weights were updated according to the Adam algorithm \cite{kingma2014adam} (available in the \textit{PyTorch} library) with respect to the $\mathscr{L}^1$-loss function and the weights were initialized following \cite{glorot2010understanding}. We implemented the autoencoders in Python whereby the code made available in \cite{davidson2018hyperspherical} served as a basis. \par

The generated handwritten numbers from the approach described above can be found in Figure \ref{image_fb}. Let us consider the natural question of whether the flexibility gained by the Fisher-Bingham distribution improves on the quality of the generated numbers. For this purpose, we fitted vMF distributions to the spherical representations of the $10$ groups. The generated numbers according to the vMF distribution can be found in Figure \ref{image_vMF}. Comparing both images, one realizes that both distributions yield good results whereby the Fisher-Bingham distribution seems to represent the latent variables more accurately (compare rows $1,2$ and $6$). This is as expected since the class of Fisher-Bingham distributions is much broader. To underline this issue, we trained our model with an encoder output dimension of $4$ (resulting in the latent space $\mathcal{S}^{2}$) and plotted the corresponding latent representations in Figure \ref{image_1} and \ref{image_6}. It is obvious that both sets do not reflect a vMF distribution. However, the goodness-of-fit for the Fisher-Bingham distribution can also be questioned; a truncated uniform distribution may be more accurate. This might be an explanation for the incorrectly drawn $8$'s in Figure \ref{image_fb}. Nevertheless, these issues are beyond the scope of this paper. Moreover, with our model, there are no concerns with overlapping latent representations since different digits were learned separately. In \cite{chen2021maximum}, the authors reported difficulties in this regard (especially for low dimensions of the latent space) resulting in the problem that a generated number may not match the number represented by the probability distribution from which a random variable was drawn.

\begin{figure}
\begin{minipage}[c]{0.49\textwidth}
    \includegraphics[width=\textwidth]{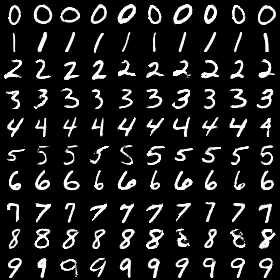}
    \caption{\label{image_fb} Fisher-Bingham distribution.}
\end{minipage}\hfill
\begin{minipage}[c]{0.49\textwidth}
    \includegraphics[width=\textwidth]{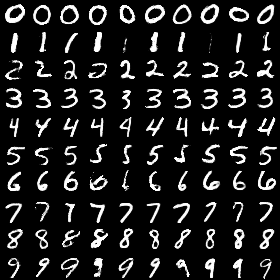}
    \caption{\label{image_vMF} von Mises-Fisher distribution.}
\end{minipage}\hfill
\end{figure}

\begin{figure}
\begin{minipage}[c]{0.49\textwidth}
    \includegraphics[width=\textwidth]{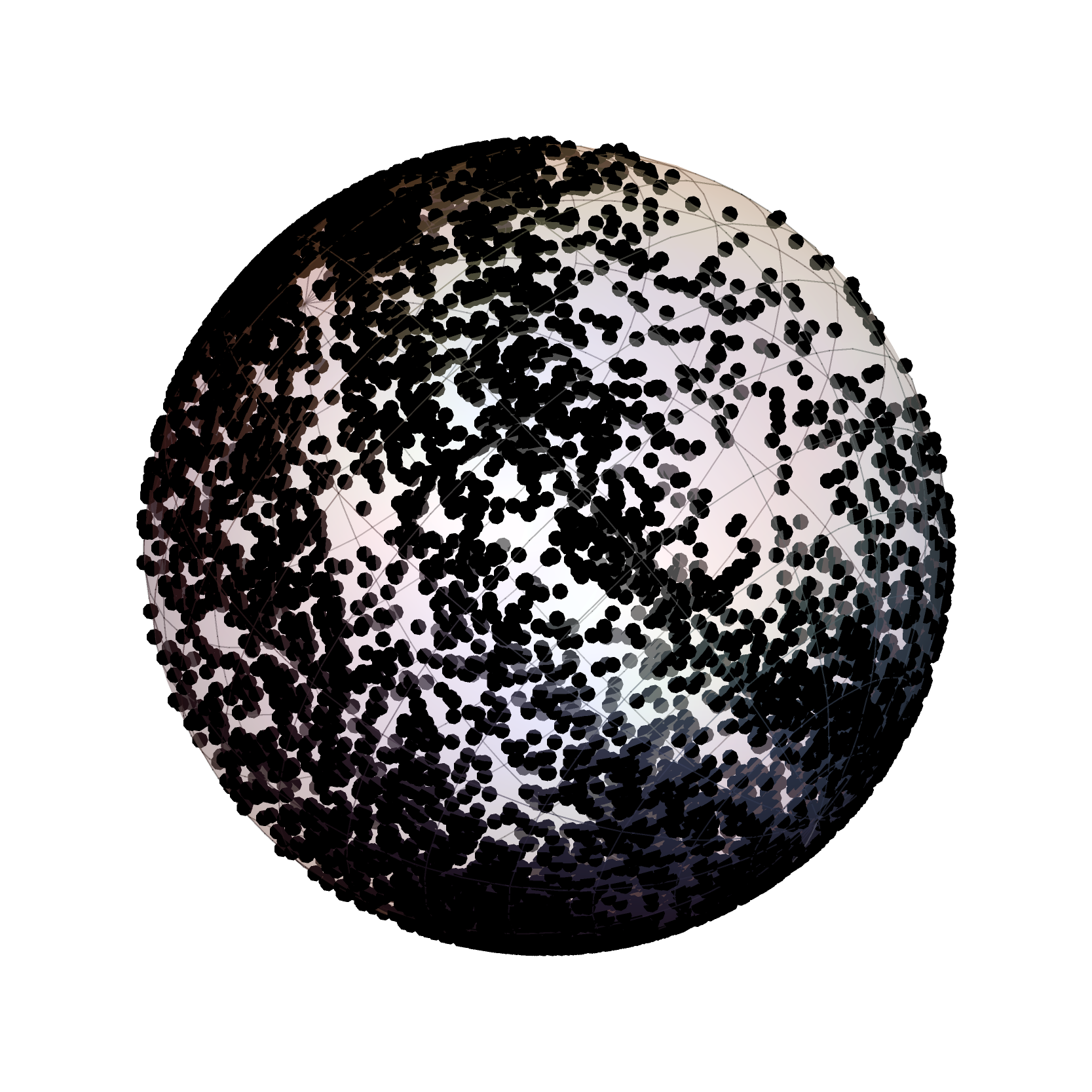}
    \caption{\label{image_1} Latent representation of the number $1$.}
\end{minipage}\hfill
\begin{minipage}[c]{0.49\textwidth}
    \includegraphics[width=\textwidth]{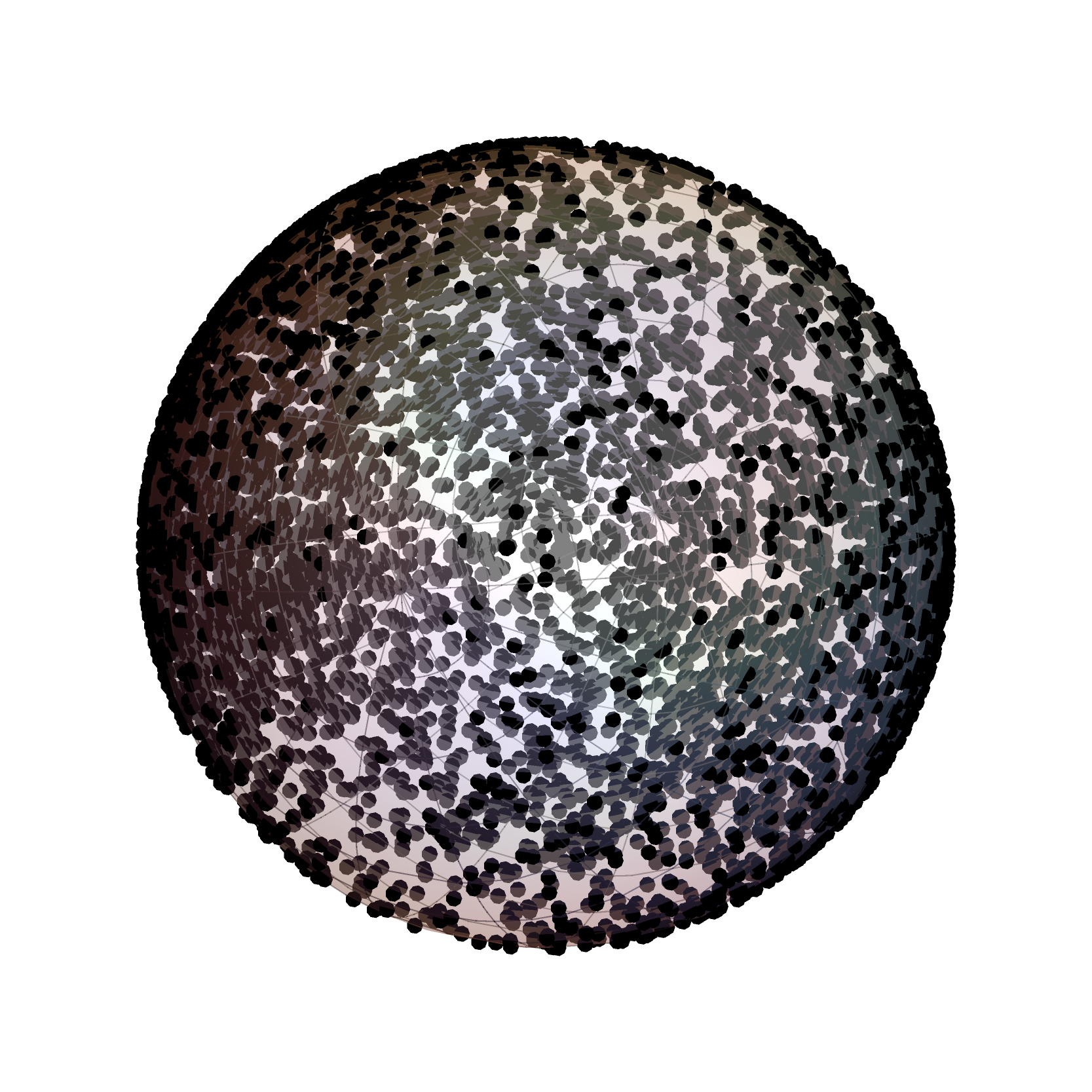}
    \caption{\label{image_6}  Latent representation of the number $6$.}
\end{minipage}\hfill
\end{figure}

\appendix
\section{Proofs} \label{appendix}

\begin{proof}[Proof of Theorem \ref{theoem_consistency_fisherbingham}]
    With \eqref{Stein_identity_manifolds} and standard matrix calculus we know that
    \begin{align*}
        \mathbb{E}[M_1'\mathrm{vech}'(A_0)+ E_1\mu_0-D_1]&=0, \\
         \mathbb{E}[G_1'\mathrm{vech}'(A_0)+ L_1\mu_0-H_1]&=0,
    \end{align*}
    and therefore together with Assumption \ref{Assumption_consistency_fisherbingham} we can write
    \begin{align*}
    \mathrm{vech}'(A_0)&= \mathbb{E}[M_1']^{-1}(\mathbb{E}[D_1]- \mathbb{E}[E_1]\mu_0), \\
    \mu_0&= (\mathbb{E}[L_1] - \mathbb{E}[G_1'] \mathbb{E}[M_1']^{-1}\mathbb{E}[E_1])^{-1} (\mathbb{E}[H_1] - \mathbb{E}[G_1'] \mathbb{E}[M_1']^{-1} \mathbb{E}[D_1]).
\end{align*}
By the strong law of large numbers we know that all sample means $M_n',D_n, E_n, G_n',H_n,L_n$ converge to their corresponding expectations. The remainder of the proof then follows by the continuous mapping theorem and Assumption \ref{Assumption_consistency_fisherbingham}.
\end{proof}

\begin{proof}[Proof of Theorem \ref{theorem_asympt_cov_fishervonmises}]
Let $X \sim vM\! F(\mu_0,\kappa_0)$. In \cite[p.\,169]{mardia2000directional}, the mean $\mathbb{E}[X]=\mu_0\mathcal{I}_{d/2}(\kappa_0)/\mathcal{I}_{d/2-1}(\kappa_0)$ is calculated. In \cite{hillen2017moments}, the authors give formulas for the covariance matrix of $X$ for dimensions $d=2$ and $d=3$. By adjusting their arguments for general dimensions and extending their results for higher moments of $X$ we obtain the following results: \vspace{0.2cm} \\ 
One has
\begin{align*}
    \mathbb{E}[XX^{\top}] = \frac{1}{\kappa_0} \frac{\mathcal{I}_{d/2}(\kappa_0)}{\mathcal{I}_{d/2-1}(\kappa_0)} I_d +  \frac{\mathcal{I}_{d/2+1}(\kappa_0)}{\mathcal{I}_{d/2-1}(\kappa_0)} \mu_0 \mu_0^{\top}.
\end{align*}
Moreover, tedious calculations give
\begin{align*}
    \mathbb{E}[\mathrm{vec}(XX^{\top})&\mathrm{vec}(XX^{\top})^{\top}]  \\
    =&\frac{\mathcal{I}_{d/2+2}(\kappa_0)}{\kappa_0\mathcal{I}_{d/2-1}(\kappa_0)}  \big( \mathrm{vec}(I_d)\mathrm{vec}(\mu_0\mu_0^{\top})^{\top} + \mathrm{vec}(\mu_0\mu_0^{\top})\mathrm{vec}(I_d)^{\top} + \mu_0\mu_0^{\top} \otimes I_d + I_d \otimes \mu_0\mu_0^{\top} \\
    & \qquad \qquad \qquad + (\mu_0\mu_0^{\top} \otimes I_d) \mathcal{K}_{d,d} + (I_d \otimes \mu_0\mu_0^{\top})\mathcal{K}_{d,d} \big) \\
     & +\frac{\mathcal{I}_{d/2+1}(\kappa_0)}{\kappa_0^2\mathcal{I}_{d/2-1}(\kappa_0)}  \big( \mathrm{vec}(I_d)\mathrm{vec}(I_d)^{\top} + I_d \otimes I_d + (I_d \otimes I_d) \mathcal{K}_{d,d} \big) \\
     & +\frac{\mathcal{I}_{d/2+3}(\kappa_0)}{\mathcal{I}_{d/2-1}(\kappa_0)}   \mathrm{vec}(\mu_0\mu_0^{\top})\mathrm{vec}(\mu_0\mu_0^{\top})^{\top},
\end{align*}
where $\mathcal{K}_{d,d}$ is the commutation matrix, as well as
\begin{align*}
    \mathbb{E}[X\mathrm{vec}(XX^{\top})^{\top}] = &\frac{\mathcal{I}_{d/2+1}(\kappa_0)}{\kappa_0\mathcal{I}_{d/2-1}(\kappa_0)}  \big( I_d \otimes \mu_0^{\top} + ( I_d \otimes \mu_0^{\top})\mathcal{K}_{d,d} + \mu_0 \, \mathrm{vec}(I_d)^{\top} \big) \\
     &+\frac{\mathcal{I}_{d/2+2}(\kappa_0)}{\mathcal{I}_{d/2-1}(\kappa_0)}  \mu_0 \, \mathrm{vec}(\mu_0\mu_0^{\top})^{\top}.
\end{align*}

We also need the corresponding covariance matrices:
\begin{align*}
    \mathrm{Var}[X] = \frac{1}{\kappa_0} \frac{\mathcal{I}_{d/2}(\kappa_0)}{\mathcal{I}_{d/2-1}(\kappa_0)} I_d + \bigg( \frac{\mathcal{I}_{d/2+1}(\kappa_0)}{\mathcal{I}_{d/2-1}(\kappa_0)} - \bigg(\frac{\mathcal{I}_{d/2}(\kappa_0)}{\mathcal{I}_{d/2-1}(\kappa_0)} \bigg)^2 \bigg) \mu_0 \mu_0^{\top}
\end{align*}
and
\begin{align*}
    \mathrm{Var}[\mathrm{vec}(XX^{\top})]  =&\frac{\mathcal{I}_{d/2+2}(\kappa_0)}{\kappa_0\mathcal{I}_{d/2-1}(\kappa_0)}  \big( \mathrm{vec}(I_d)\mathrm{vec}(\mu_0\mu_0^{\top})^{\top} + \mathrm{vec}(\mu_0\mu_0^{\top})\mathrm{vec}(I_d)^{\top} + \mu_0\mu_0^{\top} \otimes I_d  \\
    & \qquad \qquad \qquad  + I_d \otimes \mu_0\mu_0^{\top} + (\mu_0\mu_0^{\top} \otimes I_d) \mathcal{K}_{d,d} + (I_d \otimes \mu_0\mu_0^{\top})\mathcal{K}_{d,d} \big) \\
     &+\frac{\mathcal{I}_{d/2+1}(\kappa_0)}{\kappa_0^2\mathcal{I}_{d/2-1}(\kappa_0)}  \big( \mathrm{vec}(I_d)\mathrm{vec}(I_d)^{\top} + I_d \otimes I_d + (I_d \otimes I_d) \mathcal{K}_{d,d} \big) \\
     &+\frac{\mathcal{I}_{d/2+3}(\kappa_0)}{\mathcal{I}_{d/2-1}(\kappa_0)}   \mathrm{vec}(\mu_0\mu_0^{\top})\mathrm{vec}(\mu_0\mu_0^{\top})^{\top} \\
     &- \bigg(\frac{1}{\kappa_0} \frac{\mathcal{I}_{d/2}(\kappa_0)}{\mathcal{I}_{d/2-1}(\kappa_0)} \bigg)^2 \mathrm{vec}(I_d)\mathrm{vec}(I_d)^{\top} \\
     &- \frac{2}{\kappa_0}  \frac{\mathcal{I}_{d/2+1}(\kappa_0)}{\mathcal{I}_{d/2-1}(\kappa_0)} \frac{\mathcal{I}_{d/2}(\kappa_0)}{\mathcal{I}_{d/2-1}(\kappa_0)}  \mathrm{vec}(I_d)\mathrm{vec}(\mu_0\mu_0^{\top})^{\top} \\
     &-  \bigg( \frac{\mathcal{I}_{d/2+1}(\kappa_0)}{\mathcal{I}_{d/2-1}(\kappa_0)} \bigg)^2  \mathrm{vec}(\mu_0\mu_0^{\top})\mathrm{vec}(\mu_0\mu_0^{\top})^{\top}
\end{align*}
as well as
\begin{align*}
    V:=\mathbb{E}[X\mathrm{vec}&(XX^{\top})^{\top}] - \mathbb{E}[X] \mathbb{E}[\mathrm{vec}(XX^{\top})]^{\top}  \\
     =&\frac{\mathcal{I}_{d/2+1}(\kappa_0)}{\kappa_0\mathcal{I}_{d/2-1}(\kappa_0)}  \big( I_d \otimes \mu_0^{\top} + ( I_d \otimes \mu_0^{\top})\mathcal{K}_{d,d} + \mu_0 \, \mathrm{vec}(I_d)^{\top} \big) \\
     &+\frac{\mathcal{I}_{d/2+2}(\kappa_0)}{\mathcal{I}_{d/2-1}(\kappa_0)}  \mu_0 \, \mathrm{vec}(\mu_0\mu_0^{\top})^{\top} \\
     &- \bigg(\frac{\mathcal{I}_{d/2}(\kappa_0)}{\mathcal{I}_{d/2-1}(\kappa_0)} \bigg)  \bigg(\frac{1}{\kappa_0} \frac{\mathcal{I}_{d/2}(\kappa_0)}{\mathcal{I}_{d/2-1}(\kappa_0)} \mu_0 \, \mathrm{vec}(I_d)^{\top} +  \frac{\mathcal{I}_{d/2+1}(\kappa_0)}{\mathcal{I}_{d/2-1}(\kappa_0)} \mu_0 \, \mathrm{vec}(\mu_0 \mu_0^{\top})^{\top} \bigg)  .
\end{align*}

With
\begin{align*}
    Y_n=\begin{pmatrix} \mathrm{vec}(\overline{XX^{\top}}) \\
     \overline{X}\end{pmatrix}
\end{align*}
we know by the central limit theorem that
\begin{align*}
    \sqrt{n}(Y_n - \mathbb{E}[Y_n]) \overset{D}{\longrightarrow} N(0,\Lambda)
\end{align*}
with
\begin{align*}
    \Lambda = \begin{pmatrix} \mathrm{Var}[\mathrm{vec}(XX^{\top})] & V^{\top} \\ V & \mathrm{Var}[X]  \end{pmatrix}
\end{align*}
as $n \rightarrow \infty$. With $\ell:\mathbb{R}^d \rightarrow \mathbb{R}^d$, $x \mapsto x/ \Vert x \Vert$ we define $G: \mathbb{R}^{d \times d}  \times \mathbb{R}^{d} \supset \widetilde{D} \rightarrow \mathbb{R}$ as the function 
\[G(Z,z)=\frac{(d-1)\ell(z)^{\top}(I_d-Z)z}{\ell(z)^{\top}(I_d-Z)^2\ell(z)},\] 
where $Z \in \mathbb{R}^{d \times d}$, $z \in \mathbb{R}^{d}$, and $\widetilde{D} $ is the set of all $(Z,z)$ such that $\ell(z)^{\top}(I_d-Z)^2\ell(z) \neq 0$, which is an open set. In the sequel, when we differentiate a matrix-valued function with respect to a matrix-valued argument, we consider the vectorized function and the vectorized argument, i.e.\ $\frac{\partial f}{\partial x} = \frac{\partial \mathrm{vec}(f)}{\partial \mathrm{vec}(x)}$. We have that
    \begin{align} \label{fisher_von_mises_first_deriv}
    \begin{split}
        \frac{\partial G(Z,z)} {\partial Z}= &-\frac{(d-1)(z \otimes \ell(z))^{\top}}{\ell(z)^{\top} (I_d-Z)^2\ell(z)} \\
        &+ \frac{(d-1)\mathrm{vec}\big((I_d-Z)^{\top}\ell(z)\ell(z)^{\top} + \ell(z)\ell(z)^{\top}(I_d-Z)^{\top}\big)^{\top} \ell(z)^{\top} (I_d -Z ) z }{\big(\ell(z)^{\top} (I_d-Z)^2\ell(z)\big)^2}
    \end{split}
    \end{align}
    and
    \begin{align} \label{fisher_von_mises_sec_deriv}
    \begin{split} 
        \frac{\partial G(Z,z)} {\partial z}= &2(d-1) \bigg( \frac{\ell(z)^{\top}(I_d-Z)}{\ell(z)^{\top} (I_d - Z)^2 \ell(z) } - \frac{\ell(z)^{\top}(I_d-Z)\ell(z)\ell(z)^{\top}(I_d-Z)^2}{\big(\ell(z)^{\top} (I_d - Z)^2 \ell(z) \big)^2} \bigg) \big(I_d - \ell(z) \ell(z)^{\top} \big) \\
        &+ G(Z,z) \frac{\ell(z)^{\top}}{\Vert z \Vert}.
    \end{split}
    \end{align}
    For the latter derivative we applied the product rule to 
    \[G(Z,z)=\frac{(d-1)\ell(z)^{\top}(I_d-Z)\ell(z)}{\ell(z)^{\top}(I_d-Z)^2\ell(z)} \cdot \Vert z \Vert\]
    and used that $\nabla \ell(z)=(I_d -\ell(z) \ell(z)^{\top})/\Vert z \Vert $. By setting $Z=\mathbb{E}[XX^{\top}]$ and $z=\mathbb{E}[X]$ in \eqref{fisher_von_mises_first_deriv} and \eqref{fisher_von_mises_sec_deriv} we obtain
    \begin{align*}
        P_1= \frac{\mathcal{I}_{d/2-1}(\kappa_0) \kappa_0^2}{(d-1)\mathcal{I}_{d/2}(\kappa_0) } (\mu_0 \otimes \mu_0)^{\top}
    \end{align*}
    and 
    \begin{align*}
        P_2= \frac{\mathcal{I}_{d/2-1}(\kappa_0) \kappa_0}{\mathcal{I}_{d/2}(\kappa_0) } \mu_0^{\top}.
    \end{align*}
    For the calculation above we used that $\ell(\mathbb{E}[X])=\mu_0$. Now, we can apply the delta method to the random vector $Y_n$ with the function $G$. Then the asymptotic covariance matrix is given by
    \begin{align*}
        P= &P_1  \mathrm{Var}[\mathrm{vec}(XX^{\top})] P_1^{\top} + 2 P_2 V P_1^{\top} + P_2 \mathrm{Var}[X] P_2^{\top}.
    \end{align*}
    Tedious calculations then give the variance structure from the statement of the theorem.
\end{proof}

\begin{proof}[Proof of Theorem \ref{theorem_watson_asyp_norm}]
    Let $G$ be the function which maps $Y_n$ onto $\hat{\kappa}_n$. It is clear that $\hat{\kappa}_n^{\bullet}$ is a differentiable function of $Y_n$ (see \cite[Lemma 3.1.3]{kollo2005advanced} for the differentiability of $\hat{\mu}_n^{\bullet}$ with respect to $\overline{XX^{\top}}$) for a given value of $\bullet$. Now, note that under the assumptions of the theorem, $\hat{\theta}_n=(\hat{\kappa}_n,\hat{\mu}_n)$ is consistent in the sense of Theorem \ref{theoem_consistency_fisherbingham}, and therefore the probability of choosing the correct value for $\bullet$ is converging to $1$. In a similar way, with the assumptions made in the statement of the theorem, the probability that the two largest resp.\ smallest eigenvalues of the scatter matrix are distinct is converging to $1$. Hence, the function $G$ is well-defined with probability converging to $1$ and we can apply the delta method together with the central limit theorem which then gives the result.
\end{proof}

\section*{Acknowledgements}
AF is funded in part by ARC Consolidator grant from ULB and FNRS Grant CDR/OL J.0197.20 as well as EPSRC Grant EP/T018445/1. RG is funded in part by EPSRC grant EP/Y008650/1. YS is funded in part by ARC Consolidator grant from ULB and FNRS Grant CDR/OL J.0197.20.

\bibliography{library}
\bibliographystyle{abbrv}

\end{document}